\documentclass[12pt]{amsart}
\usepackage{cases}
\usepackage{mathrsfs}
\usepackage{color}
\usepackage{latexsym,amssymb}
\usepackage{a4wide}
\usepackage{amsthm}
\usepackage{amsmath}
\usepackage{graphicx}
\input xy
\xyoption{all}
\usepackage{verbatim}
\usepackage[lite,abbrev,msc-links]{amsrefs}
\usepackage{bm}

\numberwithin{equation}{section}
\begin{document}
	\theoremstyle{plain}
	\newtheorem{thm}{Theorem}[section]
	\newtheorem{lem}[thm]{Lemma}
	\newtheorem{cor}[thm]{Corollary}
	\newtheorem{cor*}[thm]{Corollary*}
	\newtheorem{prop}[thm]{Proposition}
	\newtheorem{prop*}[thm]{Proposition*}
	\newtheorem{conj}[thm]{Conjecture}	
	\theoremstyle{definition}
	\newtheorem{construction}{Construction}
	\newtheorem{notations}[thm]{Notations}
	\newtheorem{question}[thm]{Question}
	\newtheorem{prob}[thm]{Problem}
	\newtheorem{rmk}[thm]{Remark}
	\newtheorem{remarks}[thm]{Remarks}
	\newtheorem{defn}[thm]{Definition}
	\newtheorem{claim}[thm]{Claim}
	\newtheorem{assumption}[thm]{Assumption}
	\newtheorem{assumptions}[thm]{Assumptions}
	\newtheorem{properties}[thm]{Properties}
	\newtheorem{exmp}[thm]{Example}
	\newtheorem{comments}[thm]{Comments}
	\newtheorem{blank}[thm]{}
	\newtheorem{observation}[thm]{Observation}
	\newtheorem{defn-thm}[thm]{Definition-Theorem}
	\newtheorem*{Setting}{Setting}

	\newcommand{\sA}{\mathscr{A}}	
	\newcommand{\sB}{\mathscr{B}}
	\newcommand{\sC}{\mathscr{C}}
	\newcommand{\sD}{\mathscr{D}}
	\newcommand{\sE}{\mathscr{E}}
	\newcommand{\sF}{\mathscr{F}}
	\newcommand{\sG}{\mathscr{G}}
	\newcommand{\sH}{\mathscr{H}}
	\newcommand{\sI}{\mathscr{I}}
	\newcommand{\sJ}{\mathscr{J}}
	\newcommand{\sK}{\mathscr{K}}
	\newcommand{\sL}{\mathscr{L}}
	\newcommand{\sM}{\mathscr{M}}
	\newcommand{\sN}{\mathscr{N}}
	\newcommand{\sO}{\mathscr{O}}
	\newcommand{\sP}{\mathscr{P}}
	\newcommand{\sQ}{\mathscr{Q}}
	\newcommand{\sR}{\mathscr{R}}
	\newcommand{\sS}{\mathscr{S}}
	\newcommand{\sT}{\mathscr{T}}
	\newcommand{\sU}{\mathscr{U}}
	\newcommand{\sV}{\mathscr{V}}
	\newcommand{\sW}{\mathscr{W}}
	\newcommand{\sX}{\mathscr{X}}
	\newcommand{\sY}{\mathscr{Y}}	\newcommand{\sZ}{\mathscr{Z}}
	\newcommand{\bZ}{\mathbb{Z}}
	\newcommand{\bN}{\mathbb{N}}
	\newcommand{\bQ}{\mathbb{Q}}
	\newcommand{\bC}{\mathbb{C}}
	\newcommand{\bR}{\mathbb{R}}
	\newcommand{\bH}{\mathbb{H}}
	\newcommand{\bD}{\mathbb{D}}
	\newcommand{\bE}{\mathbb{E}}
	\newcommand{\bV}{\mathbb{V}}
	\newcommand{\cV}{\mathcal{V}}
	\newcommand{\cF}{\mathcal{F}}
	\newcommand{\bfM}{\mathbf{M}}
	\newcommand{\bfN}{\mathbf{N}}
	\newcommand{\bfX}{\mathbf{X}}
	\newcommand{\bfY}{\mathbf{Y}}
	\newcommand{\spec}{\textrm{Spec}}
	\newcommand{\dbar}{\bar{\partial}}
	\newcommand{\ddbar}{\partial\bar{\partial}}
	\newcommand{\redref}{{\color{red}ref}}
	\title[Koll\'ar's package for polystable parabolic Higgs bundles] {Koll\'ar's package for polystable locally abelian parabolic Higgs bundles}
	
	\author[Junchao Shentu]{Junchao Shentu}
	\email{stjc@ustc.edu.cn}
	\address{School of Mathematical Sciences, University of Science and Technology of China, Hefei, 230026, China}
	\author[Chen Zhao]{Chen Zhao}
	\email{czhao@ustc.edu.cn}
	\address{School of Mathematical Sciences, University of Science and Technology of China, Hefei, 230026, China}
	\begin{abstract}
		We generalize Koll\'ar's package (including torsion freeness, injectivity theorem, vanishing theorem and decomposition theorem) to polystable locally abelian parabolic Higgs bundles twisted by a multiplier ideal sheaf associated with an $\bR$-divisor. This gives a uniform treatment for various kinds of Koll\'ar's package in different topics in complex geometry. As applications, the weakly positivity (in the sense of Viehweg) and the generic vanishing property for higher direct image sheaves are deduced.
	\end{abstract}
	\maketitle
	\section{Introduction}
	Everything is defined over the complex number field $\bC$.
	Let $f:X\rightarrow Y$ be a proper surjective morphism from a projective variety $X$ to a complex space $Y$. We say that a coherent sheaf $\sF$ on $X$ satisfies \emph{Koll\'ar's package} with respect to $f$ if the following statements hold.
	\begin{description}
		\item[Torsion Freeness] $R^qf_\ast (\sF)$ is torsion free for every $q\geq0$ and vanishes if $q>\dim X-\dim Y$.
		\item[Injectivity Theorem] If $L$ is a semi-positive (Definition \ref{defn_semipositive_divisor}) holomorphic line bundle on $X$ so that $L^{\otimes l}$ admits a nonzero holomorphic global section $s$ for some $l>0$, then the canonical morphism
		$$R^qf_\ast(\times s):R^qf_\ast(\sF\otimes L^{\otimes k})\to R^qf_\ast(\sF\otimes L^{\otimes (k+l)})$$
		is injective for every $q\geq0$ and every $k\geq1$.
		\item[Vanishing Theorem] If $Y$ is a projective algebraic variety and $L$ is an ample line bundle on $Y$, then
		$$H^q(Y,R^pf_\ast(\sF)\otimes L)=0,\quad \forall q>0,\forall p\geq0.$$
		\item[Decomposition Theorem] $Rf_\ast (\sF)$ splits in $D(Y)$ the derived category of $\sO_Y$-modules, that is, 
		$$Rf_\ast(\sF)\simeq \bigoplus_{q} R^qf_\ast(\sF)[-q]\in D(Y).$$
		As a consequence, the spectral sequence
		$$E^{pq}_2:H^p(Y,R^qf_\ast(\sF))\Rightarrow H^{p+q}(X,\sF)$$
		degenerates at the $E_2$ page.
	\end{description}
	These statements date back to J. Koll\'ar \cite{Kollar1986_1,Kollar1986_2}, who proved that the dualizing sheaf $\omega_X$ satisfies Koll\'ar's package when $X$ is smooth and $Y$ is projective. Koll\'ar's results have since been generalized in two directions to aim for various geometric applications.
	
	The first direction is Koll\'ar's package for the dualizing sheaf twisted by a $\bQ$-divisor, or more generally, a multiplier ideal sheaf. 
	This particular type of Koll\'ar's package has significant applications in various areas of research. For example, E. Viehweg's work on the quasi-projective moduli of polarized manifolds \cite{Viehweg1995,Viehweg2010}, O. Fujino's project on the minimal model program for log-canonical varieties \cite{Fujino2017}, and the Koll\'ar-Kov\'acs' splitting criterion for du Bois singularities \cite{Kollar2010} all make use of this type of Koll\'ar's package. K. Takegoshi \cite{Takegoshi1995} has also established a proof of Koll\'ar's package for the dualizing sheaf twisted by a Nakano semi-positive vector bundle. In addition, S. Matsumura \cite{Matsumura2018} and Fujino-Matsumura \cite{Matsumura2021} have investigated the injectivity theorem for the dualizing sheaf twisted by a general multiplier ideal sheaf. Recently, Cao-P\u aun \cite{Paun2022,Paun2023} extended Koll\'ar's injectivity theorem to the log canonical bundle twisted by a semi-positive line bundle, thus confirming a conjecture proposed by Fujino \cite{Fujino2017}. Subsequently, Chan-Choi-Matsumura \cite{CCM2023} provided another proof and generalized Fujino's conjecture to the scenario of log canonical pairs (also refer to \cites{Matsumura2019,Matsumura2018}). However, a complete proof of Koll\'ar's package, especially the decomposition theorem, for the dualizing sheaf twisted by a multiplier ideal sheaf remains elusive.

	The other direction is to generalize Koll\'ar's package to certain Hodge-theoretic objects, such as variations of Hodge structure and Hodge modules. Let's assume that $f:X\to Y$ is a morphism between projective varieties. 
	Suppose that $\mathbb{V}$ is an $\mathbb{R}$-polarized variation of Hodge structure on some dense Zariski open subset $X^o\subset X_{\text{reg}}$ of the regular locus $X_{\text{reg}}$. In his work, M. Saito \cite{MSaito1991} constructs a coherent sheaf $S_X(\mathbb{V})$ (as the lowest Hodge piece of the Hodge module $IC_X(\mathbb{V})$) and shows that $S_X(\mathbb{V})$ satisfies Koll\'ar's package with respect to $f$. When $\mathbb{V}$ is the trivial variation of Hodge structure, $S_X(\mathbb{V})\simeq \omega_X$. Saito's work provides an affirmative answer to Koll\'ar's conjecture \cite[\S 4]{Kollar1986_2}. In addition to other deep results on Hodge modules, Koll\'ar's package for $S_X(\mathbb{V})$ proves to be instrumental in the series of works by Popa-Schnell \cite{PS2013,PS2014,PS2017}. 
	Recently, the authors of this article give an $L^2$-theoretic proof to Saito's result in \cite{SC2021_kollar}.
	
	The purpose of the present article is to demonstrate that Koll\'ar's package holds for specific subsheaves of a polystable parabolic Higgs bundle, which is twisted by a multiplier ideal sheaf associated with an $\mathbb{R}$-divisor.  
	This approach provides a unified and systematic treatment for various versions of Koll\'ar's package. Notably, even when considering the dualizing sheaf twisted by a multiplier ideal sheaf, this package yields novel results. The main arguments rely on the $L^2$-theoretic method developed by Andreotti-Vesentini \cite{AV1965} and H\"ormander \cite{Hormander1965}, as well as the non-abelian Hodge theory developed by Simpson \cite{Simpson1990} and Mochizuki \cite{Mochizuki2006,Mochizuki20072}.
	\subsection{Main results}
Let $X$ be a smooth projective variety and $D$ be a simple normal crossing divisor on $X$. Consider a locally abelian parabolic Higgs bundle $(H,\{{_E}H\}_{E\in{\rm Div}_D(X)},\theta)$ on $(X,D)$, which consists of the following data. 
\begin{itemize}
	\item A locally abelian parabolic vector bundle $(H,\{{_E}H\}_{E\in{\rm Div}_D(X)})$ with parabolic structures on $D$, where the filtration $\{{_E}H\}$ is indexed by the set ${\rm Div}_D(X)$ of $\bR$-divisors whose supports lie in $D$.
	\item A Higgs field $\theta:H|_{X\backslash D}\to H|_{X\backslash D}\otimes \Omega_{X\backslash D}$ which has regular singularities along $D$, meaning that $\theta({_E}H) \subset {_E}H \otimes \Omega_X(\log D)$ for every $E \in {\rm Div}_D(X)$.
\end{itemize}
This parabolic Higgs bundle is required to have vanishing first and second parabolic Chern classes and to be polystable with respect to an ample line bundle $A$ on $X$. Readers may refer to \S \ref{section_NAH} for the detailed notations regarding parabolic Higgs bundles.
	
The main object of study is a specific extension, denoted as $P_{E,(2)}(H)$, of $H|_{X\backslash D}$. To define this extension, let $E$ be an $\mathbb{R}$-divisor supported on $D$. We denote $\cup_{E'<E}{_{E'}}H$ as ${_{<E}}H$ .
The coherent sheaf $P_{E,(2)}(H)$ is determined by the following conditions.
	\begin{enumerate}
		\item ${_{<E}}H\subset P_{E,(2)}(H)\subset {_{E}}H$. In particular $P_{E,(2)}(H)|_{X\backslash D}=H|_{X\backslash D}$.
		\item Take $x$ to be a point on $D$, and let $(U;z_1,\dots, z_n)$ be holomorphic local coordinates on an open neighborhood $U$ of $x$ in $X$, such that $D=\{z_1\cdots z_r=0\}$. Let $D_i=\{z_i=0\}$, $i=1,\dots,r$. Let $N_i$ be the nilpotent part of the residue map ${\rm Res}_{D_i}(\theta)$ of the Higgs field along $D_i$. For any subset $I\subset\{1,\dots,r\}$, let $\{W(\sum_{i\in I}N_i)_{m}\}_{m\in\bZ}$ represent the monodromy weight filtration on ${_{E}}H|_U$ at $x$ with respect to $\sum_{i\in I}N_i$. 
			Then we have
			$$P_{E,(2)}(H)={_{<E}}H+\sum_{\emptyset\neq I\subset\{1,\dots,r\}}{_{\leq_IE}}H\cap\bigcap_{J\subset I}W(\sum_{i\in J}N_i)_{-\#(J)-1}$$
			on $U$ (please refer to \S \ref{section_L2_prolongation} for the notation ${_{\leq_IE}}H$).
	\end{enumerate}
When $E=D$, $P_{D,(2)}(H)$ is the sheaf of $L^2$-holomorphic sections with coefficients in $H$ (see Proposition \ref{prop_P(2)_L2}). Originally introduced by S. Zucker \cite{Zucker1979} on algebraic curves, this kind of construction involves Higgs bundles $H$ that arises from a variation of Hodge structure, making it a significant subject of study in $L^2$-cohomology of a variation of Hodge structure. The construction by Zucker has been extended to higher dimensions by Kashiwara and Kawai \cite{Kashiwara1986}. Recently, Mochizuki \cite{Mochizuki2023} provided a further generalization of the characterization of the $L^2$ de Rham complex to $\lambda$-connections (or twistor $\sD$-modules). This advancement represents a significant step in establishing the hard Lefschetz theorem for the pushforward of a twistor $\sD$-module within the framework of K\"ahler geometry. It is noteworthy that our construction $P_{D,(2)}(H)$ diverges from those presented in \cite{Zucker1979,Kashiwara1986,Mochizuki2023}, as we take into account the $L^2$-sections associated with a Hermitian metric on $X$, while other authors focus on Poincar\'e-type metrics.

To generalize Zucker's construction \cite{Zucker1979} to higher-dimensional bases and non-canonical indexed extensions, we have the sheaf $P_{E,(2)}(H)$. In particular, $P_{D-E,(2)}(H)$ combines elements of both $P_{D,(2)}(H)$ and the multiplier ideal sheaf associated with $E$ when $E\geq0$. This aspect makes $P_{E,(2)}(H)$ more convenient in applications when $E\neq D$. It can be proven that $P_{E,(2)}(H)$ is always locally free (Proposition \ref{prop_P(2)_locally_free}).
	
According to the non-abelian Hodge theory of Simpson \cite{Simpson1988, Simpson1990} and Mochizuki \cite{Mochizuki2006, Mochizuki20071}, there exists a $\mu_A$-polystable regular parabolic flat bundle $(V, \{{_E}V\}_{E\in{\rm Div}_D(X)}, \nabla)$ associated with $(H,\{{_E}H\}_{E\in{\rm Div}_D(X)},\theta)$. Furthermore, there is an isomorphism between the $C^\infty$ complex bundles:
$$\rho:H|_{X\backslash D}\otimes_{\sO_{X\backslash D}}\sC^\infty_{X\backslash D}=V|_{X\backslash D}\otimes_{\sO_{X\backslash D}}\sC^\infty_{X\backslash D} \quad (\S \ref{section_Simpson_Mochizuki}).$$
In particular, the $C^\infty$ complex bundle associated with $H|_{X\backslash D}$ has two complex structures. One is the complex structure $\dbar$ of the Higgs bundle $H|_{X\backslash D}$, and the other is the complex structure $\nabla^{0,1}$ (the $(0,1)$-part of $\nabla$) of the flat bundle $V|_{X\backslash D}$.

	The main result of the present article is the following.
	\begin{thm}\label{thm_main}	
		Let $K$ be a locally free subsheaf of $H|_{X\backslash D}$ satisfying the following conditions:
		\begin{itemize}
			\item \emph{Holomorphicity:} $\nabla^{0,1}(K) = 0$, meaning that $K$ is holomorphic with respect to both the complex structures $\dbar$ and $\nabla^{0,1}$.
			\item \emph{Weak transversality}\footnote{This condition is referred to as weak transversality due to Griffiths's transversality when $H$ arises from a variation of Hodge structure with $\{F^p\}_{p\in\bZ}$ as the Hodge filtration and $K = F^p$ for some $p$.}: $(\nabla - \theta)(K) \subset K \otimes \sA^{1,0}_{X\backslash D}$. 
		\end{itemize}
	Let $L$ be a line bundle on $X$ such that $L\simeq_{\mathbb{R}}B+N$, where $B$ is a semi-positive $\mathbb{R}$-divisor (see Definition \ref{defn_semipositive_divisor}) and $N$ is an $\mathbb{R}$-divisor on $X$ supported on $D$. Let $F$ be a Nakano semi-positive vector bundle on $X$.
	Then, the sheaf $\omega_X\otimes (P{_{D-N,(2)}}(H)\cap j_\ast K)\otimes F\otimes L$  satisfies Koll\'ar's package with respect to any proper surjective holomorphic morphism $X\to Y$ to a complex space $Y$, where $j:X\backslash D\to X$ is the immersion. 
	\end{thm}
The following examples demonstrate how Theorem \ref{thm_main} can be used to derive known results on Koll\'ar's packages with coefficients in a Hodge module or a multiplier ideal sheaf (or a combination of both).
	\subsubsection{Example: parabolic bundle}\label{example_parabolic_bundle}
	Let $X$ be a smooth projective variety and $D\subset X$ a simple normal crossing divisor on $X$. Let $(H,\{{_E}H\}_{E\in{\rm Div}_D(X)})$ be a locally abelian parabolic bundle on $(X,D)$ with vanishing first and second parabolic Chern classes, which is polystable with respect to an ample line bundle $A$ on $X$. 
	In this case, one considers $(H,\{{_E}H\}_{E\in{\rm Div}_D(X)})$ as a parabolic Higgs bundle with a vanishing Higgs field. As a result, $P_{E,(2)}(H)={_{<E}}H$. By taking $K=H|_{X\backslash D}$ in Theorem \ref{thm_main}, the holomorphicity and weak transversality conditions hold for $K$. Therefore, we obtain Koll\'ar's package with coefficients in a polystable locally abelian parabolic bundle with vanishing first and second parabolic Chern classes.
	\begin{thm}
	Let $L$ be a line bundle on $X$ such that $L\simeq_{\bR}B+N$, where $B$ is a semi-positive $\bR$-divisor and $N$ is an $\bR$-divisor on $X$ supported on $D$. Let $F$ be an arbitrary Nakano semi-positive vector bundle on $X$. 
		Then, $\omega_X\otimes {_{<D-N}}H\otimes F\otimes L$ satisfies Koll\'ar's package with respect to any proper, surjective holomorphic morphism $X\to Y$ to a complex space $Y$.
	\end{thm}
	\subsubsection{Example: twisted Koll\'ar-Saito's $S$-sheaf}\label{example_twisted_Ssheaf}
	Let $\bV$ be a variation of Hodge structure on a regular Zariski open subset of a projective variety $X$. Koll\'ar \cite{Kollar1986_2} introduced a coherent sheaf $S_X(\bV)$ that generalizes the dualizing sheaf. He conjectured that $S_X(\bV)$ satisfies Koll\'ar's package. This conjecture was later proven by Saito \cite{MSaito1991} using the theory of mixed Hodge modules. In \cite{SC2021_kollar}, the authors provide a new proof of Koll\'ar's conjecture using the $L^2$-method. Theorem \ref{thm_main} allows us to extend Koll\'ar's conjecture to $S_X(\bV)$ twisted by a multiplier ideal sheaf.
	
	Let $X$ be a projective variety, and let $\bV=(\cV,\nabla,\{\cV^{p,q}\},Q)$ be a polarized complex variation of Hodge structure (Definition \ref{defn_CVHS}) on a dense regular Zariski open subset $X^o$ of $X$. For an $\bR$-Cartier divisor $N$ on $X$, we define a coherent sheaf $S_X(\bV,-N)$ on $X$ (see \S \ref{section_Twisted_Saito}) with the following properties:
		\begin{enumerate}
		\item $S_X(\bV,0)$ is canonically isomorphic to Koll\'ar-Saito's $S_X(\bV)$ (see \cite{MSaito1991}). For an $\bR$-Cartier divisor $N\geq 0$ on $X$, $S_X(\bV,-N)$ is a combination of $S_X(\bV)$ and the multiplier ideal sheaf associated with $N$.
		\item When $X$ is smooth and $X\backslash X^o$ is a simple normal crossing divisor such that ${\rm supp}(N)\subset X\backslash X^o$, we have $S_X(\bV,-N)\simeq\omega_X\otimes (P_{D-N,(2)}(H)\cap j_\ast S(\bV))$, where $(H,\{{_E}H\}_{E\in{\rm Div}_D(X)},\theta)$ is the parabolic Higgs bundle associated with $\bV$. Here $j:X\setminus D\rightarrow X$ is the immersion and $S(\bV)$ is the top indexed nonzero Hodge bundle.
	\end{enumerate} 
	As a consequence of Theorem \ref{thm_main}, we obtain the following.
	\begin{thm}\label{thm_main_CVHS}
		Let $f:X\to Y$ be a proper surjective holomorphic morphism to a complex space $Y$. Let $L$ be a line bundle on $X$ such that $L\simeq_{\bR}B+N$ where $B$ is a semi-positive $\bR$-Cartier divisor and $N$ is an $\bR$-Cartier divisor on $X$. Let $F$ be an arbitrary Nakano semi-positive vector bundle on $X$. Then $S_{X}(\bV,-N)\otimes F\otimes L$ satisfies Koll\'ar's package with respect to $f$.
	\end{thm}
By setting $F = L = N=\mathcal{O}_X$, one can establish a proof for Koll\'ar's conjecture \cite[\S 5]{Kollar1986_2}.
    \begin{rmk}\label{rmk_Kahler_setting}
    The projectivity condition of $X$ in Theorem \ref{thm_main} can be relaxed to the condition of the existence of a tame harmonic metric on $H|_{X\backslash D}$, as proven in Theorem \ref{thm_Simpson-Mochizuki_harmonic_metric}. Therefore, when $X$ is a compact K\"ahler space, Theorem \ref{thm_main_CVHS} remains valid. This is because a polarized variation of Hodge structure allows for a tame harmonic metric, also known as the Hodge metric.
    \end{rmk}
	\subsubsection{Example: multiplier Grauert-Riemenschneider sheaf}
	When $\mathbb{V} = \mathbb{C}_{X_{\mathrm{reg}}}$ represents the trivial variation of Hodge structure and $N$ is an $\mathbb{R}$-Cartier divisor on $X$, the sheaf $\mathcal{K}_X(-N):= S_X(\mathbb{C}_{X_{\mathrm{reg}}}, -N)$\footnote{The sheaf $\mathcal{K}_X(-N)$ has been mentioned in the Nadel vanishing theorem on complex spaces \cite{Demailly2012}. When $X$ is smooth, it is referred to as the multiplier ideals by Viehweg \cite{Viehweg1995,Viehweg2010}.} is the Grauert-Riemenschneider sheaf twisted by the multiplier ideal sheaf (see \S \ref{section_GR_multiplier}) associated with $N$ when $N \geq 0$. 
	In fact, when $N = 0$, $\mathcal{K}_X(0)$ is the Grauert-Riemenschneider sheaf of $X$. 
	If $X$ is smooth and $N \geq 0$, then
	$\mathcal{K}_X(-N) \simeq \omega_X \otimes \sI(-N)$, 
	where $\sI(-N)$ is the multiplier ideal sheaf associated with $N$.
	
According to Theorem \ref{thm_main_CVHS}, we have the following theorem.
	\begin{thm}\label{thm_main_dualizing_sheaf}
		Let $f: X \to Y$ be a proper, surjective holomorphic morphism from a projective variety $X$ to a complex space $Y$. Let $L$ be a line bundle on $X$ such that $L \simeq_{\mathbb{R}} B + N$, where $B$ is a semi-positive $\mathbb{R}$-Cartier divisor, and $N$ is an $\mathbb{R}$-Cartier divisor on $X$. Let $F$ be an arbitrary Nakano semi-positive vector bundle on $X$. Then $\mathcal{K}_X(-N) \otimes F \otimes L$ satisfies Koll\'ar's package with respect to $f$.
	\end{thm}
Theorem \ref{thm_main_dualizing_sheaf} is applicable to Koll\'ar's package of pluricanonical bundles.
	\begin{cor}\label{cor_pluri_canonical_bundle}
		Let $f:X\to Y$ be a proper surjective holomorphic morphism from a smooth projective variety $X$ to a complex space $Y$. Let $A$ be a semi-positive line bundle on $X$ and $V\subset H^0(X,\omega^{\otimes km}_X\otimes A^{-1})$ a linear series for some positive integers $k$ and $m$.
		Let $F$ be an arbitrary Nakano semi-positive vector bundle on $X$. Then $\omega^{\otimes k+1}_X\otimes\sI(\frac{1}{m}|V|)\otimes F$ satisfies Koll\'ar's package with respect to $f$. 
	\end{cor}
    \begin{rmk}
    	For the same reason as in Remark \ref{rmk_Kahler_setting}, Theorem \ref{thm_main_dualizing_sheaf} holds when $X$ is a compact K\"ahler space, and Corollary \ref{cor_pluri_canonical_bundle} holds when $X$ is a compact K\"ahler manifold.
    \end{rmk}
	\subsection{Applications}
	\subsubsection{Weakly positivity of higher direct images}
By applying Viehweg's trick, we can deduce the following result from Theorem \ref{thm_main}.
	\begin{thm}\label{thm_main_positivity}
		Notations as in Theorem \ref{thm_main_CVHS}. If $X\to Y$ is a surjective morphism between smooth projective varieties, then $R^qf_\ast(\omega_{X/Y}\otimes S_{X}(\bV,-N)\otimes F\otimes L)$ is weakly positive in the sense of Viehweg \cite{Viehweg1983}.
	\end{thm}
It is worth noting that the weak positivity of $R^qf_\ast(\omega_{X/Y}\otimes F)$ has already been established by Mourougane-Takayama \cite{Takayama2009} using an alternative method.
	\subsubsection{Generic vanishing theorem}
	Together with Hacon's criterion (\cite{Hacon2004}, see also \cite{Popa2011, Schnell_GV}), one can deduce a generic vanishing theorem associated with a polystable parabolic Higgs bundle with vanishing first and second parabolic Chern classes.
	\begin{thm}\label{thm_main_GV}
	Notations as in Theorem \ref{thm_main}. Let $f:X\to A$ be a surjective morphism to an abelian variety. Then $R^qf_\ast(\omega_X\otimes(P{_{D-N,(2)}}(H)\cap j_\ast K)\otimes F\otimes L)$ is a GV-sheaf in the sense of Pareschi and Popa \cite{Popa2011}. As a consequence, 
		$${\rm codim}_{{\rm Pic}^0(A)}\left\{M\in{\rm Pic}^0(A)\mid H^i(X,\omega_X\otimes(P{_{D-N,(2)}}(H)\cap j_\ast K)\otimes F\otimes L\otimes f^\ast M)\neq0\right\}$$
		$$\geq i-(\dim X-\dim f(X)).$$
		for all $i$.
	\end{thm}
As a result, according to Theorem \ref{thm_main_CVHS}, $R^qf_\ast(\omega_X\otimes S_X(\bV,-N)\otimes F\otimes L)$ is a GV-sheaf. By setting $\bV=\bC$ as the trivial VHS, we can conclude that $R^qf_\ast(\mathcal{K}_X(-N)\otimes F\otimes L)$ is also a GV-sheaf. It is worth noting that Fujino-Matsumura \cite{Matsumura2021} have demonstrated that $R^qf_\ast(\omega_X\otimes L\otimes\sI(h))$ is a GV-sheaf for a line bundle $L$ equipped with a singular Hermitian metric $h$ with positive curvature current.
	\subsection{Organization of the article}
	In \S 2 we review the non-abelian Hodge theory established by Simpson \cite{Simpson1990} and Mochizuki \cite{Mochizuki2006,Mochizuki20072,Mochizuki20071,Mochizuki2009}. This theory serves as the primary connection between algebraic objects like polystable parabolic Higgs bundles and transcendental objects like tame harmonic bundles. It allows for the investigation of polystable parabolic Higgs bundles using the $L^2$-method.
	The proof of Theorem \ref{thm_main} is presented in \S 3. The main technical tools is the meta Koll\'ar's package (see \S 3.1) developed by the authors in \cite{SC2021_kollar} using $L^2$-methods.
\S 4 showcases examples and applications of Theorem \ref{thm_main}.
\subsection{Acknowledgement}
Both authors would like to express their sincere gratitude to Professor Takuro Mochizuki for pointing out an error in the characterization of $L^2$-sections, and for drawing their attention to the remarkable works \cite{Kashiwara1986,Mochizuki2023}. They also express appreciation to the anonymous referees, whose valuable comments contributed to the improvement of the article.

	\section{Nonabelian Hodge theory on a smooth quasi-projective variety}\label{section_NAH}
	In this section, we will review the knowledge on non-abelian Hodge theory over a smooth quasi-projective variety, as established by Simpson \cite{Simpson1990} and Mochizuki \cite{Mochizuki2006,Mochizuki20072,Mochizuki20071,Mochizuki2009}.
	
	Throughout this section, let $X$ be a smooth projective variety and $D=\sum_{i=1}^l D_i$ a reduced simple normal crossing divisor on $X$. The sheaf of rational functions on $X$ that are regular on $X\backslash D$ is denoted as $\sO_X[*D]$. Additionally, we let ${\rm Div}_D(X)$ be the $\bR$-vector space of $\bR$-divisors $E$ on $X$ where ${\rm supp}(E)\subset D$.
	Let $A=\sum_{i=1}^la_i D_i$ and $B=\sum_{i=1}^lb_i D_i$. We use the notation $A\leq (<)B$ to indicate that $a_i\leq (<)b_i$ for each $i=1,\dots,l$.
	\subsection{Parabolic Higgs bundle}
	We follow the definition given by Iyer-Simpson \cite{Simpson2007}, although the notations may differ slightly. A \emph{parabolic sheaf} ${_\ast}H$ on $(X,D)$ is a torsion free $\sO_X[*D]$-module $H$ with a collection of torsion free coherent $\sO_X$-submodules $\{{_E}H\mid E\in {\rm Div}_D(X)\}$, satisfying the following conditions:
	\begin{itemize}
		\item $H=\bigcup_{E\in {\rm Div}_D(X)}{_E}H$,
		\item ${_{E_1}}H\subset {_{E_2}}H$ if $E_1\leq E_2$,
		\item  ${_{E+\epsilon D_i}}H = {_E}H$ for any $E\in {\rm Div}_D(X)$, any $i=1,\dots,l$ and any constant $0<\epsilon\ll 1$,
		\item ${_{E+D_i}}H={_E}H\otimes \sO_X(D_i)$ for any $E\in {\rm Div}_D(X)$ and every $i=1,\dots, l$.
	\end{itemize} 
	Consequently, ${_E}H|_{X\backslash D}=H|_{X\backslash D}$ for every $E\in {\rm Div}_D(X)$. For every $E\in {\rm Div}_D(X)$, we define
	$$_{<E}H:=\bigcup_{E'<E}{_{E'}}H.$$
	\begin{defn}
		A parabolic sheaf ${_\ast}H=(H,\{{_E}H\}_{E\in{\rm Div}_D(X)})$ on $(X,D)$ is called a \emph{parabolic bundle} (resp. \emph{parabolic line bundle}) if each ${_E}H$ is a vector bundle (resp. line bundle).     
		A parabolic bundle ${_\ast}H$ on $(X,D)$ is called \emph{locally abelian} if there exists an isomorphism between ${_\ast}H$ and a direct sum of parabolic line bundles in a Zariski neighborhood of any point $x \in X$.
	\end{defn}
	\begin{defn}
		A \emph{Higgs bundle} $(H,\theta)$ on $X\backslash D$ consists of a holomorphic vector bundle $H$ on $X\backslash D$, along with an $\sO_{X\backslash D}$-linear map $\theta: H\to H\otimes \Omega_{X\backslash D}$, where $\theta$ is called a Higgs field on $H$, satisfying $\theta\wedge\theta=0$.
		
		A \emph{logarithmic Higgs bundle} $(H,\theta)$ on $(X,D)$ consists of a holomorphic vector bundle $H$ on $X$, along with an $\sO_{X}$-linear map $\theta: H\to H\otimes \Omega_{X}(\log D)$, where $\theta$ is called a Higgs field on $H$, satisfying $\theta\wedge\theta=0$.
	\end{defn}
	\begin{defn}\label{defn_parabolic_higgs_bundle}
		A \emph{locally abelian parabolic Higgs bundle} $({_\ast}H,\theta)$ on $(X,D)$ consists of a locally abelian parabolic bundle ${_\ast}H=(H,\{{_E}H\}_{E\in{\rm Div}_D(X)})$ on $(X,D)$ and a Higgs field $\theta: H|_{X\backslash D}\to H|_{X\backslash D}\otimes \Omega_{X\backslash D}$ such that $\theta$ extends to a logarithmic Higgs field
		$${_E}H\to{_E}H\otimes\Omega_X(\log D)$$
		for every $E\in{\rm Div}_D(X)$.
	\end{defn}
	\subsection{Tame harmonic bundle}
	Let $(H,\theta,h)$ be a Higgs bundle on $X\backslash D$ with a Hermitian metric $h$. Let $\overline{\theta}$ be the adjoint of $\theta$ with respect to $h$ and let $\partial$ be the unique $(1,0)$-connection such that $\partial+\dbar$ is compatible with $h$. 
	\begin{defn}
	The Higgs bundle $(H,\theta,h)$ is called a \emph{harmonic bundle} if $(\partial+\dbar+\theta+\overline{\theta})^2=0$. In this case, the Hermitian metric $h$ is called a \emph{harmonic metric}.
	\end{defn}
	Let $\nabla_h$ be the Chern connection on the harmonic bundle $H$ with respect to $h$ and let  $\Theta_h(H)=\nabla_h^2$ be its Chern curvature form. Then we have the self-dual equation
	\begin{align}\label{align_self-dual equation}
		\Theta_h(H)+\theta\wedge\overline{\theta}+\overline{\theta}\wedge\theta=0.
	\end{align}
	For the purpose of the present article, we are  focusing on  tame harmonic bundles in the sense of Simpson \cite{Simpson1990} and Mochizuki \cite{Mochizuki20072,Mochizuki20071}.
	\begin{defn}\label{defn_tame_harmonic}
		A harmonic bundle $(H,\theta,h)$ on $X\backslash D$ is called \emph{tame} if there exists a logarithmic Higgs bundle on $(X,D)$ that extends $(H,\theta)$. In this case we call $(H,\theta,h)$ a \emph{tame harmonic bundle} on $(X,D)$.
	\end{defn}
	\subsubsection{Analytic prolongations of tame harmonic bundles}
	\begin{defn}[Analytic prolongation](\cite{Mochizuki2002}, Definition 4.2)\label{defn_prolongation}
		Let $(H,h)$ be a Hermitian vector bundle on $X\backslash D$.
		Let $E\in{\rm Div}_D(X)$, $U$ be an open subset of $X$, and $s\in \Gamma(U\backslash D,H)$ be a holomorphic section. If $|s|_h=O(\prod_{i=1}^r |z_{i}|^{-a_{i}-\epsilon})$ for any positive number $\epsilon$, where $z_1,\dots,z_n$ are holomorphic local coordinates such that $E=\sum_{i=1}^ra_i\{z_i=0\}$, we denote $(s)\leq_h -E$.
		The $\sO_X$-module ${_E}H_h$ is defined as 
		$$\Gamma(U, {_E}H_h):=\{s\in\Gamma(U\backslash D,H)\mid (s)\leq_h -E\}$$
		for any open subset $U\subset X$.
	\end{defn}
Now let us recall the profound result that the set of analytic prolongations forms a locally abelian parabolic Higgs bundle (see Simpson \cite[Theorem 3]{Simpson1990} and Mochizuki \cite{Mochizuki20072}).
	\begin{thm}\label{thm_parabolic}
	 If $(H,\theta,h)$ is a tame harmonic bundle on $(X,D)$, then the bundle $ {_\ast}H_h=\left(H[\ast D],\{{_E}H_h\}_{E\in{\rm Div}_D(X)},\theta\right)$ is a locally abelian parabolic Higgs bundle. Here, $H[\ast D]=\bigcup_{E\in{\rm Div}_D(X)}{_E}H_h$ denotes the sheaf of meromorphic sections of $H$ whose poles lie in $D$.
	\end{thm}
	\subsubsection{Norm estimate}
	Consider a tame harmonic bundle $(H,\theta,h)$ on $(X,D)$. In this subsection, we will review an important norm estimate for $h$ at a point $x\in D$ given by Mochizuki in \cite{Mochizuki20072}. For convenience, we can assume that $X$ is a polydisc.
	
	Let $0\leq r\leq n$ be an integer, and let $\Delta_{\frac{1}{2}}:=\{z\in\bC\mid |z|<\frac{1}{2}\}$ be the unit disc. We can express $X$ as $\Delta^r_{\frac{1}{2}}\times\Delta^{n-r}_{\frac{1}{2}}$, with $z_1,\dots, z_n$ as the standard coordinates. Define $D_i:=\{z_i=0\}\subset X$ for $i=1,\dots,r$, $D=\cup_{i=1}^rD_i$, and $X^\ast:=X\backslash D$. Additionally, consider the tame harmonic bundle $(H,\theta,h)$ on $X^\ast$. 
	If we consider the locally abelian parabolic Higgs bundle $(H[\ast D],\{{_E}H_h\}_{E\in{\rm Div}_D(X)},\theta)$ defined by the analytic prolongations (Theorem \ref{thm_parabolic}), we can assume that the eigenvalues of the residue map ${\rm Res}_{D_i}(\theta)$ are constant on $D_i$. This assumption holds because $(H,\theta,h)$ is a restriction of a tame harmonic bundle over a compact log smooth pair.
	
	Let $E_0=\sum_{i=1}^r a_iD_i\in {\rm Div}_{D}(X)$. Suppose ${\rm Res}_{D_i}(\theta)$ has the nilpotent part $N_i: {_{E_0}}H_h|_{D_i}\to{_{E_0}}H_h|_{D_i}$. Let ${W^{(k)}}$ $(1\leq k\leq r)$ represent the monodromy weight filtration on ${_{E_0}}H_h|_{D_1\cap\cdots\cap D_k}$ associated with $N_1+\cdots+N_k$.
	\begin{thm}\emph{(\cite[\S 13.3]{Mochizuki20072})}\label{thm_tame_estimate}
		Let  $v\in {_{E_0}}H_h$ be a holomorphic local section at ${\bf 0}=(0,\dots,0)$. Assume that the following conditions hold.
		\begin{itemize}
			\item $v\notin {_{E}}H_h$ for every $E\leq E_0$, $E\neq E_0$;
			\item $v|_{\bf 0}\in \bigcap_{k=1}^r W^{(k)}_{l_k}|_{\bf 0}$ for some $l_1,\dots,l_r\in\bZ$ and $0\neq [v]\in {\rm Gr}_{l_r}^{W^{(r)}}\cdots{\rm Gr}_{l_1}^{W^{(1)}}{\rm Gr}_{E_0}H_h|_{\bf 0}$.
		\end{itemize}
		Then
		$$|v|^2_h\sim|z_1|^{-2a_1}\cdots|z_r|^{-2a_r} \left(\frac{\log|z_1|}{\log|z_2|}\right)^{l_1}\cdots\left(-\log|z_r|\right)^{l_r}$$
		over any region of the form
		$$\left\{(z_1,\dots z_n)\in (\Delta^\ast_{\frac{1}{2}})^r\times \Delta^{n-r}_{\frac{1}{2}}\bigg|\frac{\log|z_1|}{\log|z_2|}>\epsilon,\dots,-\log|z_r|>\epsilon,(z_{r+1},\dots,z_{n})\in M\right\}$$
		for any $\epsilon>0$ and an arbitrary compact subset $M\subset \Delta^{n-r}_{\frac{1}{2}}$.
	\end{thm}
	\subsection{Nonabelian Hodge theory of Simpson-Mochizuki}\label{section_Simpson_Mochizuki}
	In the context of parabolic bundles, we have the concept of parabolic Chern classes. Consequently, we can define the notion of $\mu_A$-semistable (or $\mu_A$-stable, $\mu_A$-polystable) parabolic Higgs bundles where $A$ is an ample line bundle on $X$. For a detailed explanation of these definitions, please refer to \cite[Chapter 3]{Mochizuki2006} and \cite[Definition 2.7]{Simpson2007}. 
	
Let us recall a significant result that is essential to the nonabelian Hodge theory on smooth quasi-projective varieties (see Simpson \cite{Simpson1990} and Mochizuki \cite{Mochizuki2006}).
	\begin{thm}[\cite{Simpson1990,Mochizuki2006}]\label{thm_Simpson-Mochizuki_harmonic_metric}
		Let ${_\ast}H=(H,\{{_E}H\}_{E\in{\rm Div}_D(X)},\theta)$ be a locally abelian parabolic Higgs bundle on $(X,D)$ and $A$ be an ample line bundle on $X$.
		Then the following two statements are equivalent.
		\begin{enumerate}
			\item ${_\ast}H$ is $\mu_A$-polystable with vanishing first and second parabolic Chern
			classes.
			\item There exists a tame harmonic metric $h$ on $(H|_{X\backslash D},\theta)$ such that ${_\ast}H$ is isomorphic to $\left(H[\ast D],\{{_E}H_h\}_{E\in{\rm Div}_D(X)},\theta\right)$ as a parabolic Higgs bundle.
		\end{enumerate}
	\end{thm}
    Using the tame harmonic metric, one can establish a profound relation between the following categories.
    \begin{itemize}
    \item $\sD(X,D)$: the category of $\mu_A$-polystable locally abelian parabolic flat bundles on $(X,D)$ with vanishing first and second parabolic Chern
    classes;
    \item $\sH(X,D)$: the category of $\mu_A$-polystable locally abelian parabolic Higgs bundles on $(X,D)$ with vanishing first and second parabolic Chern
    classes.
    \end{itemize}
For the purpose of this article, let us briefly explain the equivalence functor $\sS\sM:\sH(X,D)\to \sD(X,D)$.
Let ${_\ast}H=(H,\{{_E}H\}_{E\in{\rm Div}_D(X)},\theta)\in \sH(X,D)$, and let $h$ be the tame harmonic metric as in Theorem \ref{thm_Simpson-Mochizuki_harmonic_metric}.
Let $\overline{\theta}$ be the adjoint of $\theta$, and let $\partial$ be the unique $(1,0)$-connection such that $\partial+\dbar$ is compatible with $h$. Then $(H|_{X\backslash D}\otimes_{\sO_{X\backslash D}}\sA^0_{X\backslash D},\nabla:=\partial+\dbar+\theta+\overline{\theta})$ is a meromorphic flat connection that is regular along the divisor $D$, where $\sA^0_{X\backslash D}$ is the sheaf of $C^\infty$-functions on $X\backslash D$.
Let $\nabla=\nabla^{1,0}+\nabla^{0,1}$ be the decomposition according to the bidegree. Then $V_0:=\ker(\nabla^{0,1}:H|_{X\backslash D}\otimes_{\sO_{X\backslash D}}\sA^0_{X\backslash D}\to H|_{X\backslash D}\otimes_{\sO_{X\backslash D}}\sA^{0,1}_{X\backslash D})$ is a holomorphic vector bundle endowed with the flat holomorphic connection $\nabla$ and the harmonic metric $h$.
The analytic prolongations of $(V_0,h)$ (Definition \ref{defn_prolongation}) give rise to a $\mu_A$-polystable parabolic flat connection $(V:=V_0[\ast D],\{{_E}V\}_{E\in {\rm Div}_D(X)},\nabla)$ with trivial parabolic characteristic numbers. Then $\sS\sM({_\ast}H)=(V,\{{_E}V\}_{E\in {\rm Div}_D(X)},\nabla)$.
	\begin{thm}[\cite{Simpson1988,Simpson1990,Mochizuki2006,Mochizuki2009}]\label{thm_nonabelian_Hodge}
		The functor $\sS\sM:\sH(X,D)\to \sD(X,D)$ is an equivalence of categories. Furthermore, if $(H,\{{_E}H\}_{E\in{\rm Div}_D(X)},\theta)\in \sH(X,D)$ and $$\sS\sM((H,\{{_E}H\}_{E\in{\rm Div}_D(X)},\theta))=(V,\{{_E}V\}_{E\in {\rm Div}_D(X)},\nabla),$$ then there exists an identity between $C^\infty$ complex bundles $$H|_{X\backslash D}\otimes_{\sO_{X\backslash D}}\sA^0_{X\backslash D}=V|_{X\backslash D}\otimes_{\sO_{X\backslash D}}\sA^0_{X\backslash D}.$$
	\end{thm}
	\subsection{Weighted $L^2$-prolongation}\label{section_L2_prolongation}
	Let $A=\sum_{i=1}^la_i D_i$ and $B=\sum_{i=1}^lb_i D_i$, and let $I\subset\{1,\dots,l\}$ be a subset. We use the notation $A<_I B$ to indicate that $a_i\leq b_i$ for each $i=1,\dots,l$ and $a_i< b_i$ for every $i\notin I$. For each $E\in{\rm Div}_D(X)$, we define
	$${_{<_IE}}H=\bigcup_{E'<_IE}{_{E'}}H.$$
	Suppose $A$ is an ample line bundle on $X$ and $(H,\{{_E}H\}_{E\in{\rm Div}_D(X)},\theta)$ is a $\mu_A$-polystable locally abelian parabolic Higgs bundle with vanishing first and second parabolic Chern classes. We define the coherent sheaf $P_{E,(2)}(H)$ which is determined by the following conditions:
	\begin{enumerate}
		\item ${_{<E}}H\subset P_{E,(2)}(H)\subset {_{E}}H$. 
		\item Take $x$ to be a point on $D$, and let $(U;z_1,\dots, z_n)$ be holomorphic local coordinates on an open neighborhood $U$ of $x$ in $X$, such that $D=\{z_1\cdots z_r=0\}$. Let $D_i=\{z_i=0\}$, $i=1,\dots,r$. Let $N_i$ be the nilpotent part of the residue map ${\rm Res}_{D_i}(\theta)$ of the Higgs field along $D_i$. For any subset $I\subset\{1,\dots,r\}$, let $\{W(\sum_{i\in I}N_i)_{m}\}_{m\in\bZ}$ represent the monodromy weight filtration on ${_{E}}H|_U$ at $x$ with respect to $\sum_{i\in I}N_i$. 
		Then we have
	$$P_{E,(2)}(H)={_{<E}}H+\sum_{\emptyset\neq I\subset\{1,\dots,r\}}{_{<_IE}}H\cap\bigcap_{J\subset I}W(\sum_{i\in J}N_i)_{-\#(J)-1}$$
    on $U$.
	\end{enumerate}
	\begin{prop}\label{prop_P(2)_locally_free}
	The sheaf	$P_{E,(2)}(H)$ is a locally free $\sO_X$-module for every $E\in{\rm Div}_D(X)$.
	\end{prop}
	\begin{proof}
		Since ${_{<E}}H\subset P_{E,(2)}(H)\subset {_{E}}H$, one can see that $P_{E,(2)}(H)|_{X\backslash D}=H|_{X\backslash D}$ is a locally free $\sO_{X\backslash D}$-module. 	
		Let $x\in D$. According to Mochizuki \cite[Corollary 4.47, Theorem 12.48]{Mochizuki20072}, there exists a set of holomorphic local frames $e_1,\dots,e_m$ of ${_{E}}H$ that are compatible with the filtrations $\{{_{E'}}H\}_{E'\leq E}$ and the monodromy weight filtrations $W(\sum_{i\in I}N_i)_\bullet$ with $I\subset\{1,\dots,r\}$. In particular, $P_{E,(2)}(H)$, as an intersection of the filtrations, is freely generated by a subset of $e_1,\dots,e_m$. This proves that the sheaf $P_{E,(2)}(H)$ is locally free at $x$.
	\end{proof}
	\section{Proof of the main theorem}
	\subsection{A meta Koll\'ar's package}
	In this section we recall  Koll\'ar's package established in \cite{SC2021_kollar}.
Let $X$ be an irreducible complex space of dimension $n$, and let $X^o\subset X_{\rm reg}$ be a dense Zariski open subset. Consider a Hermitian vector bundle $(E,h)$ on $X^o$.
The $\sO_X$-module $S_X(E,h)$ can be defined as follows. Let $U\subset X$ be an open subset. Then $S_X(E,h)(U)$ is the space of holomorphic $E$-valued $(n,0)$-forms $\alpha$ on $U\cap X^o$ such that for every point $x\in U$, there exists a neighborhood $V_x$ of $x$ such that 
$$\int_{V_x\cap X^o}\alpha\wedge_h\overline{\alpha}<\infty.$$
	\begin{lem}\emph{(Functoriality \cite[Proposition 2.5]{SC2021_kollar})}\label{lem_L2ext_birational}
	If $\pi:X'\to X$ be a proper holomorphic map between complex spaces which is biholomorphic over $X^o$, then $$\pi_\ast S_{X'}(\pi^\ast E,\pi^\ast h)=S_X(E,h).$$
	\end{lem}
	\begin{lem}\emph{(\cite[Lemma 2.6]{SC2021_kollar})}\label{lem_Kernel}
	If $(F,h_F)$ be a Hermitian vector bundle on $X$ (in particular $h_F$ is smooth on $X$), then
		$$S_X(E,h)\otimes F\simeq S_X(E\otimes F|_{X^o},h\otimes h_{F}|_{X^o}).$$
	\end{lem}
	\begin{defn}\label{defn_tame_Hermitian_bundle}
	A Hermitian vector bundle $(E,h)$ is called \emph{tame on $X$} if, for every point $x\in X$, there is an open neighborhood $U$ containing $x$, a proper bimeromorphic morphism $\pi:\widetilde{U}\to U$ which is biholomorphic over $U\cap X^o$, and a Hermitian vector bundle $(Q,h_Q)$ on $\widetilde{U}$ such that 
		\begin{enumerate}
			\item $\pi^\ast E|_{\pi^{-1}(X^o\cap U)}\subset Q|_{\pi^{-1}(X^o\cap U)}$ as a subsheaf.
			\item There is a Hermitian metric $h'_Q$ on $Q|_{\pi^{-1}(X^o\cap U)}$ so that $h'_Q|_{\pi^\ast E}\sim \pi^\ast h$ on $\pi^{-1}(X^o\cap U)$ and
			\begin{align}\label{align_tame}
				(\sum_{i=1}^r\| f_i\|^2)^ch_Q\lesssim h'_Q
			\end{align}
			for some $c\in\bR$. Here $\{f_1,\dots,f_r\}$ is an arbitrary set of local generators of the ideal sheaf defining $\widetilde{U}\backslash \pi^{-1}(X^o)\subset \widetilde{U}$.
		\end{enumerate}
	\end{defn}
The tameness condition (\ref{align_tame}) is independent of the choice of the set of local generators. In this article, we construct a tame Hermitian vector bundle $(E,h)$ as a subsheaf of a tame harmonic bundle. In this case, Condition (\ref{align_tame}) is derived from the asymptotic behavior of the harmonic metric (Theorem \ref{thm_tame_estimate}).
	\begin{thm}\label{thm_abstract_Kollar_package}
		Let $f: X \rightarrow Y$ be a proper, locally K\"ahler morphism between irreducible complex spaces. Let $X^o \subset X_{\rm reg}$ be a dense Zariski open subset, and $(E,h)$ be a Hermitian vector bundle on $X^o$ with Nakano semi-positive curvature. If $(E,h)$ is tame on $X$, then the sheaf $S_X(E,h)$ is a coherent sheaf and satisfies Koll\'ar's package with respect to $f:X\to Y$.
	\end{thm}
    \begin{proof}
    	See \cite[Proposition 2.9 and \S 4]{SC2021_kollar}.
    \end{proof}
	\subsection{Preliminary for $\bR$-divisors}
	In this section, we will clarify some concepts regarding $\mathbb{R}$-divisors that should be familiar to experts.
	Let $X$ be a projective variety. An $\mathbb{R}$-Cartier divisor is a formal sum $A=\sum_{i=1}^r a_iA_i$, where $A_1,\dots, A_r$ are integral Cartier divisors and $a_1,\dots,a_r \in \mathbb{R}$. Given two $\mathbb{R}$-Cartier divisors, $A$ and $B$, we say that they are $\mathbb{R}$-linearly equivalent, denoted by $A \simeq_{\mathbb{R}} B$, if there exist rational functions $f_1,\cdots,f_m$ and real numbers $r_1,\dots,r_m$ such that $A-B = \sum_{i=1}^m r_i(f_i)$, where $(f_i)$ represents the principal divisor defined by $f_i$.
	Now, let $L$ be a line bundle and $A$ be an $\mathbb{R}$-Cartier divisor. We also use the notation $L \simeq_{\mathbb{R}} A$ when $(s) \simeq_{\mathbb{R}} A$ for some global rational section $s$ of $L$.
	\begin{defn}\label{defn_semipositive_divisor}
		Let $A=\sum_{i=1}^ra_iA_i$ be an $\bR$-Cartier divisor on $X$. 
		A $C^\infty$ Hermitian metric $h=(h_1,\dots,h_r)$ on $A$ consists of a $C^\infty$ Hermitian metric $h_i$ on $\sO_X(A_i)$ for each $i=1,\dots,r$. The curvature of $h$ is defined as $$\Theta_h(A):=\sum_{i=1}^ra_i\Theta_{h_i}(A_i).$$
		The $\bR$-divisor $A$ is called \emph{semi-positive} if there is a $C^\infty$ Hermitian metric $h$ on $A$ such that $\sqrt{-1}\Theta_h(A)$
		is semi-positive.
	\end{defn}
	The semipositivity of $\bR$-Cartier divisors is preserved under $\bR$-linear equivalence.
	\begin{prop}
		Let $A$ and $B$ be two $\bR$-Cartier divisors on a smooth projective variety $X$ such that $A\simeq_{\bR} B$. If  $A$ is semi-positive, then so is $B$.
	\end{prop}
	\begin{proof}
		Assuming that $A=\sum_{i=1}^ra_iA_i$, we have $h_i$ as $C^\infty$ Hermitian metrics on $\sO_X(A_i)$ for $i=1,\dots,r$ such that the curvature  $\sqrt{-1}\Theta_h(A):=\sqrt{-1}\sum_{i=1}^ra_i\Theta_{h_i}(A_i)$ is semi-positive. Now, let $B=\sum_{i=1}^sb_iB_i$ and $h'_i$ be an arbitrary $C^\infty$ Hermitian metric on $\sO_X(B_i)$ for $i=1,\dots,s$. Since $A\simeq_{\bR} B$, it follows that $[\sqrt{-1}\Theta_h(A)]=[\sqrt{-1}\Theta_h(B)]\in H^{1,1}(X,\bC)$, where $\sqrt{-1}\Theta_h(B)=\sqrt{-1}\sum_{i=1}^sb_i\Theta_{h'_i}(B_i)$. 
		According to the $\ddbar$-lemma, there exists a $C^\infty$ function $\varphi$ on $X$ such that $\Theta_h(A)=\Theta_h(B)+\ddbar(\varphi)$. By defining $h_{B_1}:=e^{-\varphi}h'_1$ and $h_{B_i}:=h'_i$ for $i=2,\dots,s$, we can show that $B$ is semi-positive.
	\end{proof}
A direct consequence is as follows: Let $L$ be a line bundle and $s$ be a global rational section of $L$. Then, $L$ is semi-positive if and only if $(s)$ is semi-positive.

Now, consider $A=\sum_{i=1}^r a_i A_i$, an $\mathbb{R}$-Cartier divisor on $X$. For each $i$, let $h_i$ be an arbitrary $C^\infty$ Hermitian metric on $\mathcal{O}_X(A_i)$. We define
\begin{align*}
	\varphi_A:=\sum_{i=1}^r a_i \log|s_i|_{h_i},
\end{align*}
where $s_i\in H^0(X,\mathcal{O}_X(A_i))$ is a defining section of $A_i$.
	\begin{defn}\label{defn_weight_A}
		 Such $\varphi_A$ is called a weight function associated with $A$.
	\end{defn}
Given different $C^\infty$ Hermitian metrics $h'_i$ and defining sections $s'_i$, there exists a constant $C>0$ such that
\begin{align}\label{align_weight_A}
	\sum_{i=1}^r a_i \log|s'_i|_{h'_i}-C \leq \sum_{i=1}^r a_i \log|s_i|_{h_i} \leq \sum_{i=1}^r a_i \log|s'_i|_{h'_i}+C.
\end{align}
   It follows that the quasi-isometric class\footnote{Two functions $f$ and $g$ are called quasi-isometric if there exists a positive constant $C$ such that $C^{-1}f\leq g\leq Cf $.} of $\exp(-2\varphi_A)$ is independent of the choice of $h_i$ and $s_i$.
    
    In the end of this section we discuss singular Hermitian metrics associated with an $\bR$-divisor.
    \begin{lem}\label{lem_singular_metric_R_div}
    	Let $A$ and $B$ be two $\mathbb{R}$-Cartier divisors on a smooth projective variety $X$. Let $L$ be a line bundle on $X$ such that $L \simeq_{\mathbb{R}} A+B$. Let $\varphi_A$ be a weight function associated with $A$ and let $h_B$ be a $C^\infty$ Hermitian metric on $B$. Then, there exists a singular Hermitian metric $h$ on $L$ that satisfies the following conditions:
    	\begin{enumerate}
    	 \item $h$ is smooth over $X\backslash{\rm supp}(A)$;
    	\item $\Theta_{h}(L|_{X\backslash{\rm supp}(A)})=\Theta_{h_B}(B)|_{X\backslash{\rm supp}(A)}$;
    	\item $|e|_h\sim\exp(-\varphi_A)$ for  a local generator $e$ of $L$.
    	\end{enumerate}
    \end{lem}
    \begin{proof}
    	Let $A = \sum_{i=1}^r a_iA_i$ and $h_A = (h_1, \dots, h_r)$ be a smooth Hermitian metric on $A$. Let $s_i \in H^0(X, \sO_X(A_i))$ be the defining section of $A_i$. We can assume that $\varphi_A = \sum_{i=1}^r a_i \log|s_i|^2_{h_i}$. Let $h_0$ be a smooth Hermitian metric on $L$. Since $L \simeq_{\mathbb{R}} A+B$, we have
    	$$[\sqrt{-1}\Theta_{h_0}(L)] = [\sqrt{-1}\Theta_{h_A}(A) + \sqrt{-1}\Theta_{h_B}(B)] \in H^{1,1}(X,\mathbb{C}).$$
    	According to the $\partial\bar{\partial}$-lemma, there exists a smooth function $\psi$ on $X$ such that
    	$$\Theta_{h_0}(L) = \Theta_{h_A}(A) + \Theta_{h_B}(B) + \partial\bar{\partial}(\psi).$$
    	Then the metric $h = \exp(\psi-\varphi_A)h_0$ satisfies all the required conditions in the lemma.
    \end{proof}
	\subsection{Proof of Theorem \ref{thm_main}}
	The main strategy in proving Theorem \ref{thm_main} is to give an $L^2$ interpretation of $\omega_X\otimes (P{_{D-N,(2)}}(H)\cap j_\ast K)\otimes F\otimes L$ (Corollary (\ref{cor_L2_interpretation})) and use Theorem \ref{thm_abstract_Kollar_package}. The remaining section will provide the details of the proof.
	\subsubsection{}\label{section_setting}
	Let $X$ be a smooth projective variety and $D$ a reduced simple normal crossing divisor on $X$. Let $(H,\{{_E}H\}_{E\in{\rm Div}_D(X)},\theta)$ be a locally abelian parabolic Higgs bundle on $(X,D)$ with trivial first and second parabolic Chern classes, which is polystable with respect to an ample line bundle $A$ on $X$. Let $h$ be a tame harmonic metric on $H|_{X\backslash D}$, which is compatible with the parabolic structure. The existence of such a metric is ensured by Theorem \ref{thm_Simpson-Mochizuki_harmonic_metric}. Let $\overline{\theta}$ be the adjoint of $\theta$, and let $\partial$ be the unique $(1,0)$-connection such that $\partial+\dbar$ is compatible with $h$. Then $(H|_{X\backslash D}\otimes_{\sO_{X\backslash D}}\sA^0_{X\backslash D},\nabla:=\partial+\dbar+\theta+\overline{\theta})$ determines a meromorphic flat connection, which is regular along the divisor $D$. Let $\nabla=\nabla^{1,0}+\nabla^{0,1}$ be the decomposition with respect to the bi-degree. Notice that $\nabla^{1,0}=\partial+\theta$ and $\nabla^{0,1}=\dbar+\overline{\theta}$.
	
	Let $K\subset H|_{X\backslash D}$ be a locally free subsheaf such that $\nabla^{0,1}(K)=0$ and $(\nabla-\theta)(K)\subset K\otimes\sA^{1,0}_{X\backslash D}$.
	\begin{lem}\label{lem_semipositive}
		$\sqrt{-1}\Theta_{h}(K)$ is Nakano semi-positive.
	\end{lem}
	\begin{proof}
		The assumptions imply that $\partial(K) = (\nabla-\theta)(K) \subset K \otimes \sA^{1,0}_{X \backslash D}$. Hence the second fundamental form of $K \subset H|_{X \backslash D}$ vanishes and we have $\Theta_{h}(K) = \Theta_{h}(H|_{X \backslash D})|_K$. Moreover, since $\overline{\theta}(K) = \nabla^{0,1}(K) - \dbar(K) = 0$, we can use the curvature formula (\ref{align_self-dual equation}) to obtain:
		\begin{align*}
			\sqrt{-1}\Theta_{h}(K) = -\sqrt{-1}\theta\wedge\overline{\theta}|_K - \sqrt{-1}\overline{\theta}\wedge\theta|_K = -\sqrt{-1}\overline{\theta}\wedge\theta|_K,
		\end{align*}
		which is Nakano semi-positive.
	\end{proof}
	Let $(F,h_F)$ be an arbitrary Nakano semi-positive Hermitian vector bundle on $X$. Let $L$ be a line bundle on $X$ such that $L\simeq_{\mathbb{R}}B+N$, where $B$ is a semi-positive $\mathbb{R}$-divisor and $N$ is an $\mathbb{R}$-divisor on $X$ which is supported on $D$. Let $\varphi_N$ be a weight function associated with $N$. By Lemma \ref{lem_singular_metric_R_div}, there is a singular Hermitian metric $h_L$ on $L$ such that the following conditions hold:
	\begin{enumerate}
		\item $h_L$ is smooth over $X\backslash{\rm supp}(N)$.
		\item \begin{align}\label{align_L_0}
			\sqrt{-1}\Theta_{h_L}(L|_{X\backslash{\rm supp}(N)})=\sqrt{-1}\Theta_{h_B}(B)|_{X\backslash{\rm supp}(N)}\geq0;
		\end{align} 
		\item   \begin{align}\label{align_norm_est_L}
			|e|_{h_L}\sim\exp(-\varphi_N)
		\end{align} for  a local generator $e$ of $L$. 
	\end{enumerate} 
    Notice that the metrics $h$ and $h_L$ are $C^\infty$ over $X\backslash D$.
	\begin{lem}\label{lem_Nsp_tame}
	The Hermitian vector bundle	$(K\otimes F|_{X\backslash D}\otimes L|_{X\backslash D},h h_Fh_L)$ is Nakano semi-positive on $X\backslash D$ and tame on $X$.
	\end{lem}
	\begin{proof}
		The first claim follows from Lemma \ref{lem_semipositive}, (\ref{align_L_0}), and the fact that $(F, h_F)$ is Nakano semi-positive. To prove the second claim, we embed $K \otimes F|_{X\setminus D} \otimes L|_{X\setminus D}$ into ${_E}H|_{X\setminus D} \otimes F|_{X\setminus D} \otimes L|_{X\setminus D}$ for some $E \in {\rm Div}_D(X)$. We still need to demonstrate that 
		$$|z_1 \cdots z_r|^c hh_F h_L \lesssim h_0$$
		for some $c > 0$ and some $C^\infty$ Hermitian metric $h_0$ on ${_E}H\otimes F\otimes L$. Here, $z_1, \ldots, z_n$ are holomorphic local coordinates on $X$ such that $D = \{z_1 \cdots z_r=0\}$.
		This is a direct consequence of Theorem \ref{thm_tame_estimate} and (\ref{align_norm_est_L}).
	\end{proof}
\subsubsection{}
To illustrate the asymptotic behavior of a section of $P_{D-E,(2)}(H)$, we require the following lemma.
\begin{lem}\label{lem_integral}
	Let $\epsilon>0$ and let $$S_{\epsilon}=\left\{(z_1,\dots z_r)\in (\Delta^\ast_{\frac{1}{2}})^r\bigg|\frac{\log|z_1|}{\log|z_2|}>\epsilon,\dots,-\log|z_r|>\epsilon\right\}.$$
	Let ${\rm vol}=dz_1\wedge d\bar{z}_1\wedge\cdots\wedge dz_r\wedge d\bar{z}_r$. Let $a_1,\dots,a_r\in\bR$ and $l_1,\dots,l_r\in\bZ$. 
	Let \begin{align*}
		r_0=\begin{cases}
			\max\{i=1,\dots,r \mid a_1 = \cdots = a_i = 1\}, & \text{if  } a_1 = 1 \\0, & \text{if } a_1 < 1.
		\end{cases}
	\end{align*}
	Then the integration
	$$\int_{S_\epsilon}|z_1|^{-2a_1}\cdots|z_r|^{-2a_r} \left(\frac{\log|z_1|}{\log|z_2|}\right)^{l_1}\cdots\left(-\log|z_r|\right)^{l_r}{\rm vol}$$
	is finite for every $\epsilon>0$ if and only if $a_1,\dots,a_r\leq 1$ and $l_i\leq -i-1$ for any $1\leq i\leq r_0$.
\end{lem}
\begin{proof}
	{\bf Step 1:}
	Let $z_i=\rho_i e^{\sqrt{-1}\theta_i}$, $i=1,\dots,r$. 
	Then
	${\rm vol}=\rho_1\cdots\rho_rd\rho_1\wedge d\theta_1\wedge\cdots\wedge d\rho_r\wedge d\theta_r$ and
	\begin{align}
	&\int_{S_\epsilon}|z_1|^{-2a_1}\cdots|z_r|^{-2a_r} \left(\frac{\log|z_1|}{\log|z_2|}\right)^{l_1}\cdots\left(-\log|z_r|\right)^{l_r}{\rm vol}\\\nonumber
	=&(2\pi)^r\int_{S'_\epsilon}\rho_1^{1-2a_1}\cdots\rho_r^{1-2a_r} \left(\frac{\log\rho_1}{\log\rho_2}\right)^{l_1}\cdots\left(-\log\rho_r\right)^{l_r}d\rho_1\wedge\cdots\wedge d\rho_r
	\end{align}
	where
	$$S'_{\epsilon}=\left\{(\rho_1,\dots \rho_r)\in (0,\frac{1}{2})^r\bigg|\frac{\log\rho_1}{\log\rho_2}>\epsilon,\dots,-\log\rho_r>\epsilon\right\}.$$
	Let 
	$$t_i=\frac{\log\rho_i}{\log\rho_{i+1}},\quad i=1,\dots,r-1\quad\textrm{and}\quad t_r=-\log\rho_r.$$
	Then 
	$$\rho_i=e^{-t_i\cdots t_r},\quad i=1,\dots,r,$$
	and 
	$$0<\rho_i=e^{-t_i\cdots t_r}<\frac{1}{2}\quad \Leftrightarrow\quad\log 2<t_i\cdots t_r<+\infty.$$
	Therefore
	$$S'_\epsilon=\{(t_1,\dots,t_r)\in\bR^r\mid t_i>\epsilon,t_i\cdots t_r>\log 2,\forall i=1,\dots,r\}.$$
The  direct computation 
	\begin{align}
		&d\rho_1\wedge d\rho_2\wedge\cdots\wedge d\rho_r\\\nonumber
		=&de^{-t_1\cdots t_r}\wedge de^{-t_2\cdots t_r}\wedge\cdots\wedge de^{-t_r}\\\nonumber
		=&\frac{\partial}{\partial t_1}(e^{-t_1\cdots t_r})\frac{\partial}{\partial t_2}(e^{-t_2\cdots t_r})\cdots\frac{\partial}{\partial t_r}(e^{-t_r})dt_1\wedge\cdots \wedge dt_r\\\nonumber
		=&(-1)^rt_2t_3^2\cdots t_r^{r-1}e^{-t_1\cdots t_r}e^{-t_2\cdots t_r}\cdots e^{-t_r}dt_1\wedge\cdots\wedge dt_r
	\end{align}
   yields that
	\begin{align}\label{align_integral_Hodge_metric_sector}
		&\int_{S'_\epsilon}\rho_1^{1-2a_1}\cdots\rho_r^{1-2a_r} \left(\frac{\log\rho_1}{\log\rho_2}\right)^{l_1}\cdots\left(-\log\rho_r\right)^{l_r}d\rho_1\wedge\cdots\wedge d\rho_r\\\nonumber
		=&(-1)^r\int_{S'_\epsilon}\prod_{i=1}^re^{(2a_i-2)t_i\cdots t_r}t_i^{l_i+i-1}dt_1\wedge\cdots\wedge dt_r.
	\end{align}
    {\bf Step 2:} In this step, we show that $a_1,\dots,a_r\leq 1$ whenever the integral (\ref{align_integral_Hodge_metric_sector}) is finite for every $\epsilon>0$. We will proceed by contradiction. Specifically,  we assume that there exists some index $m\in \{1,\dots,r\}$ such that $a_m>1$ (and we fix such $m$). Take $0<\epsilon<1$ so that
    \begin{align}
    	c(m,\epsilon):=2a_m-2-\sum_{j=1}^{m-1}|2a_j-2|(2\epsilon)^{m-j}>0.
    \end{align}
    Consider the subset
    $$L_{m,\epsilon}=\{(t_1,\dots,t_r)\in\bR^r\mid t_1,\dots, t_{m-1}\in(\epsilon,2\epsilon);t_{m+1},\dots,t_r\in[1,2];t_m>\epsilon^{1-m}\}\subset S'_\epsilon.$$
    Then the inequality 
    \begin{align*}
    	\sum_{i=1}^r(2a_i-2)t_i\cdots t_r
    	=&\left((2a_m-2)+\sum_{j=1}^{m-1}(2a_j-2)t_j\cdots t_{m-1}\right)t_m\cdots t_r+\sum_{j=m+1}^r(2a_j-2)t_j\cdots t_r\\\nonumber
    	\geq& c(m,\epsilon)t_m-2^{r-m}\sum_{j=m+1}^r|2a_j-2|
    \end{align*}
holds on $L_{m,\epsilon}$.
    As a consequence, one has
    \begin{align}
    	&\int_{S'_\epsilon}\prod_{i=1}^re^{(2a_i-2)t_i\cdots t_r}t_i^{l_i+i-1}dt_1\wedge\cdots\wedge dt_r\\\nonumber
    	\geq&\int_{L_{m,\epsilon}}\prod_{i=1}^re^{(2a_i-2)t_i\cdots t_r}t_i^{l_i+i-1}dt_1\wedge\cdots\wedge dt_r\\\nonumber
    	\gtrsim&\int_{\epsilon^{1-m}}^{\infty}e^{c(m,\epsilon)t_m}t_m^{l_m+m-1}dt_m=+\infty
    \end{align}
    because $c(m,\epsilon)>0$.
    
    {\bf Step 3:} For the converse, we assume that $a_1,\dots,a_r\leq 1$. We aim to demonstrate that the integral (\ref{align_integral_Hodge_metric_sector}) is finite for every $\epsilon > 0$ if $l_i\leq -i-1$ for all $1\leq i\leq r_0$. 
 Given the assumption that $a_1,\dots,a_r \leq 1$ and $a_{r_0+1}<1$, we can infer from
    \begin{align*}
    	t_{r_0+1}^{l_{r_0+1}+r_0}\cdots t_r^{l_r+r-1}\lesssim e^{(1-a_{r_0+1})t_{r_0+1}\cdots t_r}\quad\textrm{where }t_{i}>\epsilon,i=r_0+1\dots,r
    \end{align*}
that 
    \begin{align}\label{align_L2lemma_1}
    	\prod_{i=r_0+1}^re^{(2a_i-2)t_i\cdots t_r}t_i^{l_i+i-1}\lesssim e^{(a_{r_0+1}-1)t_{r_0+1}\cdots t_r}\quad\textrm{where }t_{i}>\epsilon,i=r_0+1\dots,r.
    \end{align}
    As a consequence, we obtain that 
    \begin{align}\label{align_L2Lemma_2}
    	&\int_{S'_\epsilon}\prod_{i=1}^re^{(2a_i-2)t_i\cdots t_r}t_i^{l_i+i-1}dt_1\wedge\cdots\wedge dt_r\\\nonumber
    	\leq&\int_{\epsilon}^{\infty}\cdots \int_{\epsilon}^{\infty}\prod_{i=1}^re^{(2a_i-2)t_i\cdots t_r}t_i^{l_i+i-1}dt_1\cdots dt_r\\\nonumber
        =&\int_{\epsilon}^{\infty}t_1^{l_1}dt_1\cdots \int_{\epsilon}^{\infty}t_{r_0}^{l_{r_0}+r_0-1}dt_{r_0}\int_{\epsilon}^{\infty}\cdots \int_{\epsilon}^{\infty}\prod_{i=r_0+1}^re^{(2a_i-2)t_i\cdots t_r}t_i^{l_i+i-1}dt_{r_0+1}\cdots dt_r\\\nonumber
        \lesssim&\int_{\epsilon}^{\infty}t_1^{l_1}dt_1\cdots \int_{\epsilon}^{\infty}t_{r_0}^{l_{r_0}+r_0-1}dt_{r_0}\int_{\epsilon}^{\infty}\cdots \int_{\epsilon}^{\infty}e^{(a_{r_0+1}-1)t_{r_0+1}\cdots t_r}dt_{r_0+1}\cdots dt_r\quad(\ref{align_L2lemma_1}).\\\nonumber
    \end{align}
    Since $l_i+i-1\leq -2$ for every $1\leq i\leq r_0$, it follows that 
    $$\int_{\epsilon}^{\infty}t_i^{l_i+i-1}dt_i<\infty,\quad 1\leq i\leq r_0.$$
   Then the finiteness of the right hand side of (\ref{align_L2Lemma_2}) is a consequence of Lemma \ref{lem_int_1} below. 
    
    {\bf Step 4:} We assert that if $a_1, \dots, a_r \leq 1$ and $l_m > -m-1$ for some $1\leq m\leq r_0$, then the integral (\ref{align_integral_Hodge_metric_sector}) is infinite for any $0 < \epsilon < 1$.
    
    Let $$L'_{m,\epsilon}=\{(t_1,\dots,t_r)\in\bR^r\mid t_i\in(1,2),\forall i\neq m; t_m>1\}\subset S_{\epsilon}'.$$
    Then we have
    \begin{align}\label{align_L2Lemma_3}
    	&\int_{S'_\epsilon}\prod_{i=1}^re^{(2a_i-2)t_i\cdots t_r}t_i^{l_i+i-1}dt_1\wedge\cdots\wedge dt_r\\\nonumber
    	\geq&\int_{L'_{m,\epsilon}}\prod_{i=1}^re^{(2a_i-2)t_i\cdots t_r}t_i^{l_i+i-1}dt_1\wedge\cdots\wedge dt_r\\\nonumber
    	=&\int_{1}^{2}t_1^{l_1}dt_1\cdots \int_{1}^{\infty}t_{m}^{l_m+m-1}dt_m\cdots \int_{1}^{2}t_{r_0}^{l_{r_0}+r_0-1}dt_{r_0}\int_{1}^{2}\cdots \int_{1}^{2}\prod_{i=r_0+1}^re^{(2a_i-2)t_i\cdots t_r}t_i^{l_i+i-1}dt_{r_0+1}\cdots dt_r\\\nonumber
    	\gtrsim& \int_{1}^{\infty}t_{m}^{l_m+m-1}dt_m,
    \end{align}
which is infinite since $l_m+m-1\geq -1$.
\end{proof}
\begin{lem}\label{lem_int_1}
	Let $\epsilon>0$ and $c>0$ be constants. Then 
	$$\int_{\epsilon}^{\infty}\cdots \int_{\epsilon}^{\infty}e^{-ct_1\cdots t_r}dt_1\cdots dt_r<\infty.$$
\end{lem}
\begin{proof}
	We will show that 
	\begin{align}\label{align_lem_int_proof1}
		\int_{\epsilon}^{\infty}\cdots \int_{\epsilon}^{\infty}e^{-ct_1\cdots t_k}dt_1\cdots dt_k<\infty
	\end{align}
holds by induction on $k$. The initial case $k=1$ follows from 
	$$\int_{\epsilon}^{\infty}e^{-ct_1}dt_1=-\frac{1}{c}e^{-ct_1}\big|^{\infty}_{\epsilon}=\frac{1}{c}e^{-c\epsilon}.$$
	Assuming that (\ref{align_lem_int_proof1}) is valid for \( k=r-1 \) and any \( c>0 \), we proceed to show it holds for \( k=r \):
	\begin{align}
		&\int_{\epsilon}^{\infty}\cdots \int_{\epsilon}^{\infty}e^{-ct_1\cdots t_r}dt_1\cdots dt_r\\\nonumber
		=&\int_{\epsilon}^{\infty}\cdots \int_{\epsilon}^{\infty}\left(\int_{\epsilon}^{\infty}e^{-ct_1\cdots t_r}dt_r\right)dt_1\cdots dt_{r-1}\\\nonumber
		=&\int_{\epsilon}^{\infty}\cdots \int_{\epsilon}^{\infty}\frac{1}{ct_1\cdots t_{r-1}}e^{-c\epsilon t_1\cdots t_{r-1}}dt_1\cdots dt_{r-1}\\\nonumber
		\lesssim&\int_{\epsilon}^{\infty}\cdots \int_{\epsilon}^{\infty}e^{-\frac{1}{2}c\epsilon t_1\cdots t_{r-1}}dt_1\cdots dt_{r-1}\\\nonumber
		<&\infty.
	\end{align}
   Thus, the proof is concluded.
\end{proof}
\begin{prop}\label{prop_P(2)_L2}
	Let $ds^2$ be a Hermitian metric on $X$. Let $E\in{\rm Div}_D(X)$ with $\varphi_E$ as its weight function. Then a holomorphic section $s$ of $H$ lies in $P_{D-E,(2)}(H)$ if and only if $|s|^2_{e^{-\varphi_E}h}{\rm vol}_{ds^2}$ is locally integrable at every point of $D$.
\end{prop}
	\begin{proof}
		Since the problem is local, let's assume that $X=\Delta^r\times\Delta^{n-r}$ for some $0\leq r\leq n$. We have $z_1,\dots, z_n$ as its standard holomorphic coordinates, and $D=\cup_{i=1}^rD_i$, where $D_i:=\{z_i=0\}\subset X$ for $i=1,\dots,r$. Let $E=\sum_{i=1}^rk_iD_i$ with $k_1,\dots,k_r\in\bR$. We also have $E_0=\sum_{i=1}^r a_iD_i\in {\rm Div}_{D}(X)$ and $v\in {_{E_0}}H_h$ as a holomorphic local section at ${\bf 0}=(0,\dots,0)$. Suppose ${\rm Res}_{D_i}(\theta)$ has the nilpotent part $N_i: {_{E_0}}H_h|_{D_i}\to{_{E_0}}H_h|_{D_i}$. Let ${W^{(k)}}$ $(1\leq k\leq r)$ represent the monodromy weight filtration on ${_{E_0}}H_h|_{D_1\cap\cdots\cap D_k}$ associated with $N_1+\cdots+N_k$.
		
		Assume the following conditions are met:
		\begin{itemize}
			\item $v\notin {_{E'}}H_h$ for every $E'\leq E_0$, $E'\neq E_0$;
			\item $v|_{\bf 0}\in \bigcap_{k=1}^r W^{(k)}_{l_k}|_{\bf 0}$ for some $l_1,\dots,l_r\in\bZ$ and $0\neq [v]\in {\rm Gr}_{l_r}^{W^{(r)}}\cdots{\rm Gr}_{l_1}^{W^{(1)}}{\rm Gr}_{E_0}H_h|_{\bf 0}$.
		\end{itemize}
		By Theorem \ref{thm_tame_estimate}, we can estimate
		\begin{align*}
			|v|^2_{e^{-\varphi_E}h}\sim|z_1|^{-2k_1-2a_1}\cdots|z_r|^{-2k_r-2a_r} \left(\frac{\log|z_1|}{\log|z_2|}\right)^{l_1}\cdots\left(-\log|z_r|\right)^{l_r}
		\end{align*}
		over any region of the form
		$$S_{(z_1,\cdots,z_r),\epsilon}=\left\{(z_1,\dots z_n)\in (\Delta^\ast)^r\times \Delta^{n-r}\bigg|\frac{\log|z_1|}{\log|z_2|}>\epsilon,\dots,-\log|z_r|>\epsilon,(z_{r+1},\dots,z_{n})\in M\right\}$$
		for any $\epsilon>0$ and an arbitrary compact subset $M\subset \Delta^{n-r}$. Let \begin{align*}
				r_0=\begin{cases}
					\max\{i=1,\dots,r \mid k_1+a_1 = \cdots = k_i+a_i = 1\}, & \text{if  } k_1+a_1 = 1 \\0, & \text{if } k_1+a_1 < 1.
				\end{cases}
			\end{align*}
		According to Lemma \ref{lem_integral}, the integral
		$\int_{S_{(z_1,\cdots,z_r),\epsilon}}|v|^2_{e^{-\varphi_E}h}{\rm vol}_{ds^2}$ is finite near ${\bf 0}$ for every $\epsilon>0$ if and only if $$k_1+a_1,\dots,k_r+a_r\leq 1$$ and $l_i\leq -i-1$ for every $1\leq i\leq r_0$. Since this statement holds on every sector $S_{(z_{\sigma(1)},\cdots,z_{\sigma(r)}),\epsilon}$ ($\sigma\in S_r$), it follows that $|v|^2_{e^{-\varphi_E}h}$ is integral near $\bm{0}$ if and only if $v\in P{_{D-E,(2)}}(H)$. This proves the proposition.
	\end{proof}
    We denote by $j:X\backslash D\to X$ the immersion.
    \begin{cor}\label{cor_L2_interpretation}
    	$\omega_X\otimes (P{_{D-N,(2)}}(H)\cap j_\ast K)\otimes F\otimes L\simeq S_X(K\otimes F|_{X\backslash D}\otimes L|_{X\backslash D},hh_Fh_L)$. 
    \end{cor}
    \begin{proof}
    	According to Proposition \ref{prop_P(2)_L2}, we have a natural isomorphism:
    	$$\omega_X\otimes (P_{D-N,(2)}(H)\cap j_\ast K) \simeq S_X(K,e^{-\varphi_N}h).$$
    	Using Lemma \ref{lem_Kernel}, we obtain:
    	$$\omega_X\otimes (P_{D-N,(2)}(H)\cap j_\ast K)\otimes F\otimes L \simeq S_X(K,e^{-\varphi_N}h)\otimes F\otimes L \simeq S_X(K\otimes F|_{X\backslash D}\otimes L|_{X\backslash D},hh_Fh_L).$$
    \end{proof}
\begin{cor}\label{cor_locally_free}
The sheaf	$P{_{E,(2)}}(H)\cap j_\ast K$ is locally free for every $E\in{\rm Div_D(X)}$.
\end{cor}
\begin{proof}
	Since $\nabla^{0,1}(K) = 0$, it follows that
	$$\partial(K) = (\nabla-\theta)(K) \subset K\otimes\sA^{1,0}_{X\backslash D}.$$
	As a result, the second fundamental form of $K\subset H|_{X\backslash D}$ vanishes, indicating that the orthogonal complement $K^\bot$ is holomorphic. Furthermore, $H|_{X\backslash D} = K\oplus K^{\bot}$. Let $s_1$ be a section of $j_\ast K$ and $s_2$ a section of $j_\ast K^{\bot}$. Let $\varphi_{D-E}$ denote a weight function associated with $D-E$. Since $s_1$ and $s_2$ are pointwise orthogonal to each other, $|s_1+s_2|_{e^{-\varphi_{D-E}}h}$ is locally $L^2$ if and only if both $|s_1|_{e^{-\varphi_{D-E}}h}$ and $|s_2|_{e^{-\varphi_{D-E}}h}$ are locally $L^2$. According to Proposition \ref{prop_P(2)_L2}, this implies that 
	$$P{_{E,(2)}}(H) = (P{_{E,(2)}}(H)\cap j_\ast K)\oplus (P{_{E,(2)}}(H)\cap j_\ast K^\bot).$$
	Since the sheaf $P{_{E,(2)}}(H)$ is locally free (Proposition \ref{prop_P(2)_locally_free}), it follows that the sheaf $P{_{E,(2)}}(H)\cap j_\ast K$ is also locally free.
\end{proof}
Now, we are ready to prove the main theorem of this article.
\begin{thm}
	Let $f:X\to Y$ be a proper surjective holomorphic morphism  to a complex space. Then 
	$\omega_X\otimes (P{_{D-N,(2)}}(H)\cap j_\ast K)\otimes F\otimes L$ satisfies Koll\'ar's package with respect to $f$.
\end{thm}
\begin{proof}
Based on Theorem \ref{thm_abstract_Kollar_package} and Corollary \ref{cor_L2_interpretation}, it is enough to demonstrate that $(K\otimes F|_{X\backslash D}\otimes L|_{X\backslash D},hh_Fh_L)$ is Nakano semipositive and tame on $X$. This can be achieved by using Lemma \ref{lem_Nsp_tame}.
\end{proof}
	\section{Examples and applications}
	\subsection{Koll\'ar's package for Koll\'ar-Saito's $S$-sheaf twisted by an $\bR$-divisor}
	\subsubsection{Complex variation of Hodge structure}\label{section_CVHS}
	\begin{defn}{\cite[\S 8]{Simpson1988}}\label{defn_CVHS}	
		Let $X^o$ be a smooth variety. 
		A \emph{polarized complex variation of Hodge structure} on $X^o$ of weight $k$ is a flat holomorphic connection $(\cV,\nabla)$ on $X^o$ together with a decomposition $\cV\otimes_{\sO_{X^o}}\sA^0_{X^o}=\bigoplus_{p+q=k}\cV^{p,q}$ of $C^\infty$ bundles and a flat Hermitian form $Q$ on $\cV$ such that
		\begin{enumerate}
			\item The Hermitian form $h_Q$ which equals $(-1)^{p}Q$ on $\cV^{p,q}$ is a Hermitian metric on the $C^\infty$ complex vector bundle $\cV\otimes_{\sO_{X^o}}\sA^0_{X^o}$.
			\item The decomposition $\cV\otimes_{\sO_{X^o}}\sA^0_{X^o}=\bigoplus_{p+q=k}\cV^{p,q}$ is orthogonal with respect to $h_Q$.
			\item The Griffiths transversality condition 
			\begin{align}\label{align_Griffiths transversality}
				\nabla(\cV^{p,q})\subset \sA^{0,1}(\cV^{p+1,q-1})\oplus \sA^{1,0}(\cV^{p,q})\oplus\sA^{0,1}(\cV^{p,q})\oplus \sA^{1,0}(\cV^{p-1,q+1})
			\end{align}
			holds for every $p$ and $q$. Here $\sA^{i,j}(\cV^{p,q})$ denotes the sheaf of smooth $(i,j)$-forms with values in $\cV^{p,q}$.
		\end{enumerate}
	\end{defn}
Let $X$ be a smooth projective variety and $\cup_{i=1}^l D_i=D:=X\backslash X^o\subset X$ be a simple normal crossing divisor, where $D_1,\dots, D_l$ are irreducible components. Let $\bV=(\cV,\nabla,\{\cV^{p,q}\},Q)$ be a polarized complex variation of Hodge structure on $X^o:=X\backslash D$. 
Take the decomposition
$$\nabla=\overline{\theta}+\partial+\dbar+\theta$$
according to (\ref{align_Griffiths transversality}). Let $H={\rm Ker}(\dbar:\cV\otimes_{\sO_{X^o}}\sA^0_{X^o}\to \cV\otimes_{\sO_{X^o}}\sA^{0,1}_{X^o})$.
Then triple $(H,\theta,h_Q)$ is a harmonic bundle associated with $(\cV,\nabla,h_Q)$ through Simpson's correspondence \cite[\S 8]{Simpson1988}. Furthermore, there is an orthogonal decomposition of holomorphic subbundles:
$$H=\oplus_{p+q=k}H^{p,q}$$
where $H^{p,q}=H\cap\cV^{p,q}$. In addition, we have:
\begin{align*}
	\theta(H^{p,q})\subset H^{p-1,q+1}\otimes \Omega_{X^o}.
\end{align*}
	\begin{prop}\emph{(\cite[Proposition 5.4]{SC2021_kollar})}\label{prop_VHS_tame}
	The harmonic bundle	$(H,\theta,h_Q)$ is tame on $X$.
	\end{prop}
Combining the above proposition with Theorem \ref{thm_parabolic} and Theorem \ref{thm_Simpson-Mochizuki_harmonic_metric}, we can conclude that $\left(H_{\bV}:=\bigcup_{E\in{\rm Div}_D(X)}{_E}H_{h_Q},\{{_E}H_{h_Q}\}_{E\in{\rm Div}_D(X)},\theta\right)$ is a $\mu_A$-polystable locally abelian parabolic Higgs bundle on $(X, D)$. It is called the locally abelian parabolic Higgs bundle associated with $\bV$.

Let $S(\bV):={\rm ker}(\nabla^{0,1})\cap\cV^{p_{\rm max},k-p_{\rm max}}$ where $p_{\rm max}=\max\{p|\cV^{p,k-p}\neq0\}$ and 	$\nabla=\nabla^{1,0}+\nabla^{0,1}$ is the bi-degree decomposition. The following lemma can be easily derived from Definition \ref{defn_CVHS}.
	\begin{lem}\label{lem_SV_conditions}
		$\nabla^{0,1}(S(\bV))=0$, $\dbar(S(\bV))=0$ and $(\nabla-\theta)(S(\bV))\subset S(\bV)\otimes\sA^{1,0}_{X\backslash D}$.
	\end{lem}
	\subsubsection{Local weight of $S(\bV)$}
	Let $x \in D$ and $(U; z_1, \dots, z_n)$ denote holomorphic local coordinates on an open neighborhood $U$ of $x$ in $X$, such that $D = \{z_1 \cdots z_r = 0\}$. Define $D_i = \{z_i = 0\}$ for $i = 1, \dots, r$. Let $N_i$ represent the nilpotent part of the residue map $\mathrm{Res}_{D_i}(\theta)$ of the Higgs field along $D_i$. Given a subset $I \subset \{1, \dots, r\}$, let $E \in \mathrm{Div}_D(X)$ and consider $\{{_E}W(\sum_{i \in I} N_i)_m\}_{m \in \mathbb{Z}}$, the monodromy weight filtration on ${_E}H_{h_Q}|_U$ at $x$ with respect to $\sum_{i \in I} N_i$. Let $j: U \setminus D \to U$ denote the immersion. In the subsequent section, we will utilize the following semi-purity results of $S(\bV)$ to characterize $P_{E,(2)}(H_\bV) \cap j_* S(\bV)$.
	\begin{lem}\label{lem_W_F}
		Suppose $n = r = 1$, that is, $U = \Delta$ is the unit disc and $D = \{0\}$. Then we have $$\big(j_\ast S(\bV)\cap{_{E}}H_{h_Q}\big)\cap {_E}W(N_1)_{-1}|_{0}\subset {_{<E}}H_{h_Q}|_{0}.$$
	\end{lem}
    \begin{proof}
    	Let us consider the quotient ${\rm Gr}_E H_{h_Q} := {_{E}}H_{h_Q}/{_{<E}}H_{h_Q}$, equipped with the monodromy weight filtration $W(N_1) := {_E}W(N_1)\!\!\mod {_{<E}}H_{h_Q}$ and the Hodge filtration $F^\bullet$ defined by  
    	$$
    	F^p := j_\ast \cF^p \cap {_{E}}H_{h_Q}/{_{<E}}H_{h_Q},
    	$$  
    	where $\cF^\bullet$ denotes the holomorphic Hodge filtration of $\bV$. According to \cite[\S 6.8]{Schnell2025} (see also \cite[6.16]{Schmid1973}),  
    	$$
    	({\rm Gr}_E H_{h_Q}, W(N_1), F^\bullet)
    	$$  
    	forms a mixed Hodge structure. To prove the lemma, it suffices to show that $W_{-1}(N_1) = 0$.  
    	
    	Suppose $W_{-1}(N_1) \neq 0$, and let $k$ denote the weight of $\bV$. Define $l := \max\{l \mid W_{-l}(N_1) \neq 0\}$. Then $l \geq 1$. By \cite[\S 6.8]{Schnell2025}, the filtration $F^\bullet$ induces a pure Hodge structure of weight $m + k$ on $W_m(N_1)/W_{m-1}(N_1)$. Moreover, the map  
    	$$
    	N_1^l : W_l(N_1)/W_{l-1}(N_1) \to W_{-l}(N_1)/W_{-l-1}(N_1)
    	$$  
    	is an isomorphism of type $(-l, -l)$. Denote $S(\bV) = \cF^p$ for some $p$. By the definition of $l$, any nonzero element $\alpha \in W_{-l}(N_1)$ induces a nonzero class $[\alpha] \in W_{-l}(N_1)/W_{-l-1}(N_1)$ of Hodge type $(p, k-l-p)$. Since $N_1^l$ is an isomorphism, there exists $\beta \in W_l(N_1)/W_{l-1}(N_1)$ of Hodge type $(p+l, k-p)$ such that $N_1^l(\beta) = [\alpha]$. However, $\beta = 0$ because $F^{p+l} = 0$. This contradicts the fact that $[\alpha] \neq 0$. Consequently, $W_{-1}(N_1)$ must be zero. This completes the proof of the lemma.
    \end{proof}
    \begin{cor}\label{cor_W_F}
    	For every subset $I\subset\{1,\dots,r\}$, one has $$\big(j_\ast S(\bV)\cap{_{E}}H_{h_Q}\big)\cap W(\sum_{i\in I}N_i)_{-1}|_{\bf 0}\subset{_{<E}}H_{h_Q}|_{\bf 0}.$$
    \end{cor}
    \begin{proof}
    	Consider the family of curves 
    	$$\iota_{I,t}:\Delta\to U,\quad z\mapsto(\delta_{I,1}(z,t),\dots,\delta_{I,r}(z,t),0,\cdots,0),\quad t\in\Delta^\ast,$$
    	where \begin{align*}
    	\delta_{I,i}(z,t)=\begin{cases}
    		z, & \text{if  } i\in I \\t, & \text{otherwise }.
    		\end{cases}
    	\end{align*}
    	 Then the nilpotent part of the residue map of the Higgs field $\iota_{I,t}^\ast(H,\theta)$ is given by $\sum_{i\in I}N_i$. By applying Lemma \ref{lem_W_F} on $\iota_{I,t}^\ast(H,\theta)$, one has
    	$$\big(j_\ast S(\bV)\cap{_{E}}H_{h_Q}\big)\cap W(\sum_{i\in I}N_i)_{-1}\subset {_{<E}}H_{h_Q}$$
    	at each point $\iota_{I,t}(0)$ with $t\in\Delta^\ast$. Since the sheaves $j_\ast S(\bV)\cap{_{E}}H_{h_Q}$, $W(\sum_{i\in I}N_i)_{-1}$ and ${_{<E}}H_{h_Q}$ are locally free, taking $t\to 0$ yields that 
    	$$\big(j_\ast S(\bV)\cap{_{E}}H_{h_Q}\big)\cap W(\sum_{i\in I}N_i)_{-1}|_{\bf 0}\subset{_{<E}}H_{h_Q}|_{\bf 0}.$$
    \end{proof}
	\subsubsection{Twisted Koll\'ar-Saito's S-sheaf}\label{section_Twisted_Saito}
	In this section, we introduce a construction that combines Koll\'ar-Saito's $S_X(\bV)$ and the multiplier ideal sheaf associated with an $\bR$-divisor. Additionally, we generalize Koll\'ar-Saito's construction to complex variations of Hodge structure, without making assumptions on the local monodromy. This is interesting from the perspective of nonabelian Hodge theory, as complex variations of Hodge structure are precisely the $\bC^\ast$ fixed points on the moduli space of certain tame harmonic bundles (see \cite[Theorem 8]{Simpson1990}, \cite[Proposition 1.9]{Mochizuki2006}).
	
	Let $X$ be a projective variety and $X^o\subset X_{\rm reg}$ be a dense Zariski open subset. Let $\bV=(\cV,\nabla,\{\cV^{p,q}\},Q)$ be a polarized complex variation of Hodge structure on $X^o$. Let $h_Q$ be the Hodge metric associated with $Q$ and let $\left(H_{\bV}:=\bigcup_{E\in{\rm Div}_D(X)}{_E}H_{h_Q},\{{_E}H_{h_Q}\}_{E\in{\rm Div}_D(X)},\theta\right)$ be the $\mu_A$-polystable locally abelian parabolic Higgs bundle on $(X, D)$ associated with $\bV$.
	Let $N$ be an $\bR$-Cartier divisor on $X$. We define the coherent sheaf $S_X(\bV,-N)$ as follows.
	\begin{description}
		\item[Log smooth case] 
		Assuming that $X$ is smooth and $D:=X\backslash X^o$ is a simple normal crossing divisor with ${\rm supp}(N)\subset D$, we  denote the irreducible decomposition of $D$ as $D=\cup_{i=1}^lD_i$, and we  write $N=\sum_{i=1}^l{r_i}D_i$ where $r_1,\dots,r_l\in\bR$. For  $\bm{r}=(r_1,\dots,r_l)$, we let $\cV_{>\bm{r}-1}$ be the Deligne-Manin prolongation with indices $>\bm{r}-1$. This is a locally free $\sO_X$-module that extends $\cV$ in such a way that $\nabla$ induces a connection with logarithmic singularities given by
		$$\nabla:\cV_{>\bm{r}-1}\to\cV_{>\bm{r}-1}\otimes\Omega_X(\log D),$$
		where the real part of the eigenvalues of the residue of $\nabla$ along $D_i$ belongs to $(r_i-1,r_i]$ for each $i$. 
		We can now define $S_{X}(\bV,-N)$ as
		\begin{align}\label{align_S_X_P}
			S_{X}(\bV,-N)&:=\omega_X\otimes\left(j_\ast S(\bV)\cap\cV_{>\bm{r}-1}\right)\\\nonumber
			&\simeq \omega_X\otimes\left(j_\ast S(\bV)\cap{_{<D-N}}H_{h_Q}\right)\\\nonumber
			&\simeq \omega_X\otimes\big(P_{D-N,(2)}(H_{\bV})\cap j_\ast S(\bV)\big)\quad \textrm{(Corollary \ref{cor_W_F})}.
		\end{align}
	    Here $j:X^o\to X$ is the open immersion. 
		\item[General case] Let $\pi:\widetilde{X}\to X$ be a resolution of singularities such that $\pi^o:=\pi|_{\pi^{-1}(X^o\backslash{\rm supp}(N))}:\pi^{-1}(X^o\backslash{\rm supp}(N))\to X^o\backslash{\rm supp}(N)$ is biholomorphic and the exceptional loci $E:=\pi^{-1}((X\backslash X^o)\cup {\rm supp}(N)))$ is a simple normal crossing divisor. Then
		\begin{align*}
	    S_X(\bV,-N)\simeq\pi_\ast\left(S_{\widetilde{X}}(\pi^{o\ast}\bV,-\pi^\ast N)\right).
		\end{align*}
	\end{description}
	\begin{prop}\label{prop_S_X_L2}
		We have $S_{X}(\bV,-N)\simeq S_X(S(\bV),e^{-\varphi_N}h_Q)$, where $\varphi_N$ is a weight function associated with $N$. In particular, $S_{X}(\mathbb{V},-N)$ is a torsion free coherent sheaf on $X$ that is independent of the choice of the desingularization $\pi:\widetilde{X}\to X$.
	\end{prop}
	\begin{proof}
		Let $\pi:\widetilde{X}\to X$ be a resolution of singularities such that $\pi^o:=\pi|_{\pi^{-1}(X^o\backslash{\rm supp}(N))}:\pi^{-1}(X^o\backslash{\rm supp}(N))\to X^o\backslash{\rm supp}(N)$ is biholomorphic and the reduced exceptional loci $E:=\pi^{-1}((X\backslash X^o)\cup {\rm supp}(N)))$ is a simple normal crossing divisor. It follows from  (\ref{align_S_X_P}) that
		\begin{align*}
			S_{\widetilde{X}}(\pi^{o\ast}\bV,-\pi^\ast(N))&\simeq \omega_{\widetilde{X}}\otimes\big(P_{E-\pi^{\ast}N,(2)}(H_{\pi^{o\ast}\bV})\cap j_\ast S(\pi^{o\ast}\bV)\big)\\\nonumber
			&\simeq S_{\widetilde{X}}(S(\pi^{o\ast}\bV),e^{-\pi^\ast(\varphi_N)}\pi^{o\ast}h_Q)\quad(\textrm{Proposition \ref{prop_P(2)_L2}}),
		\end{align*}
	where $j:\widetilde{X}\backslash E\to\widetilde{X}$ is an immersion.
		Applying $\pi_\ast$ on both sides, we get 
		\begin{align*}
			S_{X}(\bV,-N)&\simeq\pi_\ast(S_{\widetilde{X}}(\pi^{o\ast}\bV,-\pi^\ast(N)))\\\nonumber
			&\simeq \pi_\ast(S_{\widetilde{X}}(S(\pi^{o\ast}\bV),e^{-\pi^\ast(\varphi_N)}\pi^{o\ast}h_Q))\\\nonumber
			&\simeq S_X(S(\bV),e^{-\varphi_N}h_Q)\quad \textrm{(Lemma \ref{lem_L2ext_birational})}.
		\end{align*}
	\end{proof}
	\begin{thm}
		Let $f:X\to Y$ be a proper surjective holomorphic morphism from a projective variety $X$ to a complex space $Y$. Let $\bV$ be a polarized complex variation of Hodge structure on some regular dense Zariski open subset $X^o\subset X$. Let $L$ be a line bundle on $X$ such that $L\simeq_\bR B+N$ where $B$ is a semi-positive $\bR$-Cartier divisor and $N$ is an $\bR$-Cartier divisor on $X$. Let $F$ be an arbitrary Nakano semi-positive vector bundle on $X$. Then $S_{X}(\bV,-N)\otimes F\otimes L$ satisfies Koll\'ar's package with respect to $f$.
	\end{thm}
	\begin{proof}
		Let $\pi:\widetilde{X}\to X$ be a resolution of singularities such that $\pi^o:=\pi|_{\pi^{-1}(X^o\backslash{\rm supp}(N))}:\pi^{-1}(X^o\backslash{\rm supp}(N))\to X^o\backslash{\rm supp}(N)$ is biholomorphic and the reduced exceptional loci $E:=\pi^{-1}((X\backslash X^o)\cup {\rm supp}(N)))$ is a simple normal crossing divisor. 
		Let $j: \widetilde{X}\backslash E \to \widetilde{X}$ be an immersion. Equation (\ref{align_S_X_P}) tells us that 
		\[S_{\widetilde{X}}(\pi^{o*}\mathbb{V},-\pi^* N)\simeq \omega_{\widetilde{X}}\otimes\big(P_{E-\pi^\ast N,(2)}(H_{\pi^{o\ast}\bV})\cap j_*S( \pi^{o*}\mathbb{V})\big).\]
		By applying Theorem \ref{thm_main} (torsion freeness) to $\pi:\widetilde{X}\rightarrow X$, we obtain
		\begin{align*}
			S_X(\mathbb{V},-N)\otimes F\otimes L&\simeq \pi_*\left(S_{\widetilde{X}}(\pi^{o*}\mathbb{V},-\pi^* N)\otimes\pi^* F\otimes\pi^* L\right)\\\nonumber
			&\simeq R\pi_*\left(S_{\widetilde{X}}(\pi^{o*}\mathbb{V},-\pi^* N)\otimes\pi^* F\otimes\pi^* L\right).
		\end{align*}
		Therefore
		$$Rf_\ast(S_{X}(\bV,-N)\otimes F\otimes L)\simeq R(f\circ\pi)_\ast\left(\omega_{\widetilde{X}}\otimes\big(P_{E-\pi^\ast N,(2)}(H_{\pi^{o\ast}\bV})\cap j_\ast S(\pi^{o\ast}\bV)\big)\otimes\pi^\ast F\otimes\pi^\ast L\right).$$
		By Lemma \ref{lem_SV_conditions},
		the theorem follows by applying Theorem \ref{thm_main} to $\omega_{\widetilde{X}}\otimes\big(P_{E-\pi^\ast N,(2)}(H_{\pi^{o\ast}\bV})\cap j_\ast S(\pi^{o\ast}\bV)\big)\otimes\pi^\ast F\otimes\pi^\ast L$ with respect to the morphism $f\circ\pi$.
	\end{proof}
\begin{rmk}
	Similar results hold for the $S$-sheaf twisted by a multiplier ideal sheaf associated with an ideal.
	Let $c > 0$ be a real number and $\mathfrak{a}$ be a coherent ideal sheaf on $X$. Consider a desingularization $\pi: \widetilde{X} \to X$ such that $\mathfrak{a}\sO_{\widetilde{X}}=\sO_{\widetilde{X}}(-E)$, where $E\geq0$ is a simple normal crossing divisor, and $\pi^\ast(X\backslash X^o)$ is a simple normal crossing divisor that intersects transversally with $E$.
	We define $S_X(\bV,-c\mathfrak{a})$ as $\pi_\ast(S_{\widetilde{X}}(\pi^\ast\bV,-cE))$. Theorem \ref{thm_main_CVHS} provides relevant results for $S_X(\bV,-c\mathfrak{a})$.
\end{rmk}
\subsection{Multiplier Grauert-Riemenschneider sheaf}\label{section_GR_multiplier}
When $\mathbb{V}=\mathbb{C}_{X_{\mathrm{reg}}}$ is the trivial variation of Hodge structure and $N$ is an $\mathbb{R}$-Cartier divisor on $X$, $S_X(\mathbb{C}_{X_{\mathrm{reg}}},-N)$ is precisely the Grauert-Riemenschneider sheaf twisted by the multiplier ideal sheaf associated with $N$ when $N\geq0$. This is referred to as the multiplier ideals by Viehweg \cite{Viehweg1995,Viehweg2010}, and it also appears in the Nadel vanishing theorem on complex spaces \cite{Demailly2012}. Let us briefly describe its construction for the convenience of the readers.
\begin{description}
	\item[Log smooth case]If $X$ is smooth and ${\rm supp}(N)$ is a simple normal crossing divisor, then
	$$\mathcal{K}_X(-N):=\omega_X\otimes\sO_X(-\lfloor N\rfloor).$$ 
	\item[General case] Let $\pi:\widetilde{X}\to X$ be a proper bimeromorphic morphism such that $\pi^o:=\pi|_{\pi^{-1}(X_{\rm reg}\backslash{\rm supp}(N))}:\pi^{-1}(X_{\rm reg}\backslash{\rm supp}(N))\to X_{\rm reg}\backslash{\rm supp}(N)$ is biholomorphic and the exceptional loci $E:=\pi^{-1}((X\backslash X_{\rm reg})\cup {\rm supp}(N)))$ is a simple normal crossing divisor. Then
	\begin{align*}
	\mathcal{K}_X(-N):=\pi_\ast\left(\mathcal{K}_{\widetilde{X}}(-\pi^\ast N)\right).
	\end{align*}
\end{description}
Certainly, $\mathcal{K}_X(-N)\simeq S_X(\mathbb{C}_{X_{\text{reg}}},-N)$. Proposition \ref{prop_S_X_L2} implies that $\mathcal{K}_X(-N)$ is independent of the choice of the desingularization. When $X$ is smooth and $N\geq0$, we have 
$$\mathcal{K}_X(-N)\simeq\omega_X\otimes\sI(-N)$$
where $\sI(-N)$ is the multiplier ideal sheaf associated with $N$.
According to Theorem \ref{thm_main_CVHS}, we have the following result.
\begin{thm}\label{thm_main_dualizing_sheaf_proof}
	Let $f:X\to Y$ be a proper surjective holomorphic morphism from a projective variety $X$ to a complex space $Y$. Let $L$ be a line bundle on $X$ such that $L\simeq_{\bR}B+N$ where $B$ is a semi-positive $\bR$-Cartier divisor and $N$ is an $\bR$-Cartier divisor on $X$. Let $F$ be an arbitrary Nakano semi-positive vector bundle on $X$. Then $\mathcal{K}_{X}(-N)\otimes F\otimes L$ satisfies Koll\'ar's package with respect to $f$.
\end{thm}
Theorem \ref{thm_main_dualizing_sheaf_proof} has an application in Koll\'ar's package of pluricanonical bundles.
	\begin{cor}
		Let $f:X\to Y$ be a proper surjective holomorphic morphism from a smooth projective variety $X$ to a complex space $Y$. Let $K_X$ be the canonical divisor of $X$ and $\omega_X=\sO_X(K_X)$. Suppose $A$ is a semi-positive line bundle on $X$ and $V\subset|km K_X-A|$ is a linear series for some positive integers $k$ and $m$.
		Let $F$ be an arbitrary Nakano semi-positive vector bundle on $X$. Then the tensor product $\omega_X^{\otimes (k+1)}\otimes\sI(\frac{1}{m}|V|)\otimes F$ satisfies Koll\'ar's package with respect to $f$.\\
	\end{cor}
	\begin{proof}
		Let $\mathfrak{a}$ be the ideal sheaf of the base scheme of $V$. Consider a principalization $\pi:\widetilde{X}\to X$ of the ideal sheaf $\mathfrak{a}$, which satisfies the condition that $\mathfrak{a}\sO_{\widetilde{X}}=\sO_{\widetilde{X}}(-E)$, where $E$ is a $\pi$-exceptional divisor with $E\geq0$.
		From this, we have $\pi^\ast(kmK_X-A)=B+E$, where $B$ is a semi-ample divisor. Let $L = \pi^*(\omega_X^{\otimes k})$. Then we have $L \simeq_{\mathbb{Q}} \frac{1}{m}(\pi^\ast A+B) + \frac{1}{m}E$. By the functoriality of multiplier ideal sheaves, we obtain
		$$\omega_X^{\otimes (k+1)}\otimes \sI(\frac{1}{m}|V|)\otimes F\simeq\pi_\ast(\omega_{\widetilde{X}}\otimes L\otimes \sI(\frac{1}{m}E)\otimes\pi^{\ast}F).$$
		Notice that $\frac{1}{m}(\pi^\ast A+B)$ is semi-positive. By applying Theorem \ref{thm_main_dualizing_sheaf} to $\pi$, we have
		\begin{align*}
			\omega_X^{\otimes (k+1)}\otimes \sI(\frac{1}{m}|V|)\otimes F \simeq\pi_\ast(\omega_{\widetilde{X}}\otimes L\otimes \sI(\frac{1}{m}E)\otimes\pi^{\ast}F) \simeq R\pi_\ast(\omega_{\widetilde{X}}\otimes L\otimes \sI(\frac{1}{m}E)\otimes\pi^{\ast}F).
		\end{align*}
		Therefore, we have
		$$Rf_\ast(\omega_X^{\otimes (k+1)}\otimes \sI(\frac{1}{m}|V|)\otimes F) \simeq R(f\circ\pi)_\ast\big(\omega_{\widetilde{X}}\otimes L\otimes \sI(\frac{1}{m}E)\otimes\pi^{\ast}F\big).$$
		Therefore, the corollary follows by applying Theorem \ref{thm_main_dualizing_sheaf} to $\omega_{\widetilde{X}}\otimes L\otimes \sI(\frac{1}{m}E)\otimes\pi^{\ast}F$ with respect to the morphism $f\circ\pi$.
	\end{proof}
	\subsection{Semi-positivity of higher direct images}
	Let $Y$ be a smooth projective variety and $F$ be a torsion free sheaf on $Y$. We say that $F$ is weakly positive (in the sense of Viehweg \cite{Viehweg1983}) on some Zariski open subset $U\subset X$ if, for every ample line bundle $A$ on $Y$ and every $a\in\bZ_{>0}$, there exists $b\in\bZ_{>0}$ such that $S^{[ab]}F\otimes A^b$ is generated by global sections at each point of $U$. In this context, $S^{[ab]}F$ refers to the reflexive hull of the symmetric power $S^{ab}F$. We can say that $F$ is \emph{weakly positive} if it is weakly positive over some Zariski open subset. As noted in \cite[Remark 1.3]{Viehweg1983}, it is sufficient to check this definition for a fixed line bundle $A$ (which does not necessarily need to be ample, as seen in \cite[Lemma 2.14]{Viehweg1995}).
	
	Notations as in Theorem \ref{thm_main}. Let $x\in D$ be a point. Let $D_1,\dots,D_r$ be the components of $D$ that contains $x$. For every $E\in{\rm Div}_D(X)$ and every $i=1,\dots,r$ let $N_r\in {\rm End}({_{E}}H/{_{<E}}H|_{D_i})$ be the nilpotent part of the residue map ${\rm Res}_{D_i}(\theta)$ of the Higgs field $\theta$. For every $I\subset\{1,\dots,r\}$, let $W(\sum_{i\in I}N_i)_\bullet$ be the monodromy weight filtration associated with $\sum_{i\in I}N_i$. 
	\begin{defn}
		Let $K\subset H|_{X\backslash D}$ be a locally free subsheaf. We say that $K$ \emph{has non-negative weight at $x$} if 
		$$(_{E}H\cap j_\ast(K))|_x\cap W(\sum_{i\in I}N_i)_{-1}|_x=0$$
		for every $E\in{\rm Div}_D(X)$ and every $I\subset \{1,\dots,r\}$. $K$  \emph{has non-negative weight along $D$} if it has non-negative weight at every point of $D$.
	\end{defn}
    There are two examples of $K$ that has non-negative weight along $D$. The first example is presented in Example \ref{example_parabolic_bundle}, where the weight filtrations are trivial due to the vanishing of the Higgs fields. The other is example \ref{example_twisted_Ssheaf}, wherein $H$ arises from a polarized complex variation of Hodge structure $\bV$ on \( X \backslash D \), and \( K = S(\bV) \) represents the top non-zero Hodge bundle. In this context, Corollary \ref{cor_W_F} demonstrates that \( S(\bV) \) has non-negative weight along $D$.
	\begin{thm}
		Notations as in Theorem \ref{thm_main}. Assume that $K$ has non-negative weight along $D$. If $f:X\to Y$ is a surjective morphism between smooth projective varieties, then $R^qf_\ast(\omega_{X/Y}\otimes(P{_{D-N,(2)}}(H)\cap j_\ast K)\otimes F\otimes L)$ is weakly positive.
	\end{thm}
	\begin{proof}
		Let us denote the sheaf $(P_{D-N,(2)}(H)\cap j_{\ast}K)\otimes F\otimes L$ by $\sE$ for simplicity. Additionally, let $A_Y$ be a line bundle on $Y$ that is very ample. Our goal is to demonstrate that, for every $m>0$, the sheaf
		$$S^{[m]}R^qf_{\ast}(\omega_{X/Y}\otimes \sE)\otimes \omega_Y\otimes A_Y^{\dim Y+1}$$
		is generated by global sections over a Zariski open subset of $Y$.
		
		\emph{Step 1:}
		Let $\Delta_Y$ be the set of critical values of $f$. By performing possibly blowing-ups $\rho:X'\to X$, we can assume that $\rho^{\ast}(f^\ast\Delta_Y)\cup \rho^\ast(D)$ forms a divisor with simple normal crossings. Using Lemma \ref{lem_L2ext_birational} and Corollary \ref{cor_L2_interpretation}, we have:
		$$\omega_X\otimes\sE\simeq \rho_\ast\left(\omega_{X'}\otimes (P_{\rho^\ast(D-N),(2)}(\rho^\ast H)\cap j'_\ast(\rho^\ast K))\otimes \rho^\ast F\otimes\rho^\ast L\right)$$
		where $j':X'\backslash\rho^\ast D\to X'$ is the open immersion. By applying Theorem \ref{thm_main} (torsion freeness) to the morphism $\rho:X'\to X$, we can conclude that
		$$\omega_X\otimes\sE\simeq R\rho_\ast\left(\omega_{X'}\otimes (P_{\rho^\ast(D-N),(2)}(\rho^\ast H)\cap j'_\ast(\rho^\ast K))\otimes \rho^\ast F\otimes\rho^\ast L\right).$$
		Therefore, there exists an isomorphism:
		$$R^q(f\circ\rho)_\ast\left(\omega_{X'}\otimes (P_{\rho^\ast(D-N),(2)}(\rho^\ast H)\cap j'_\ast(\rho^\ast K))\otimes \rho^\ast F\otimes\rho^\ast L\right)\simeq R^qf_\ast(\omega_X\otimes\sE)$$
		for every $q$. As a result, we can assume that $f^\ast(\Delta_Y)\cup D$ forms a divisor with a simple normal crossing on $X$ for the rest of the proof. Specifically, we assume that $f$ is flat in codimension $1$ on $Y$.
		
		\emph{Step 2: Viehweg's trick.}
		We denote the main component of the $m$-th fiber product $X\times_Y\cdots\times_YX$ by $X^{[m]}$, and let $f^{[m]}:X^{[m]}\to Y$ denote the projection map. Let $D^{[m]}:=\sum_{i=1}^{m} p_i^\ast(D)$, where $p_i:X^{[m]}\to X$ is the projection map to the $i$-th component. Consider the largest Zariski open subset $U\subset Y$ where $(X,D)$ is log smooth over $U$. Then, $f^{[m]}:(X^{[m]},D^{[m]})\to Y$ is log smooth over $U$. Let $X^{[m]o}:=(f^{[m]})^{-1}(U)$.
		Let $\pi:X^{(m)}\to X^{[m]}$ be a resolution of singularities such that $\pi^\ast(X^{[m]}\backslash X^{[m]o})$ is a simple normal crossing divisor, and $\pi$ is an isomorphism on $X^{(m)o}:=\pi^\ast(X^{[m]o})$.
		
		According to Corollary \ref{cor_locally_free}, $\sE$ is a locally free $\sO_X$-module. Let $hh_Fh_L$ be the singular Hermitian metric on $\sE$ as introduced in \S\ref{section_setting}. Given that $K$ has non-negative weight along $D$, it follows from Theorem \ref{thm_tame_estimate} that there exists a $C^\infty$ Hermitian metric $h_0$ on $\sE$ such that $h_0 \lesssim hh_Fh_L$.
		We define $\sE^{[m]}:=\otimes_{i=1}^{m}p_i^\ast\sE$ and $\sE^{(m)}:=\pi^{\ast}\sE^{[m]}$. Let $V\subset Y$ be a Zariski open subset such that $R^qf_\ast(\omega_X\otimes \sE)$ is locally free on $V$ for every $q\geq0$, and $f$ is flat over $V$. Since $R^qf_\ast(\omega_X\otimes \sE)$ is torsion free for each $q\geq0$ (Theorem \ref{thm_main}), we can assume that $Y\backslash V$ is of codimension $\geq 2$ for simplicity.
		
		Since $f:X\to Y$ is a Gorenstein morphism, we can conclude that $f^{[m]}:X^{[m]}\to Y$ is also a Gorenstein morphism. Moreover, there exists an isomorphism
		\begin{align*}
			\omega_{X^{[m]}/Y}\simeq \bigotimes_{i=1}^{m} p_i^\ast\omega_{X/Y}
		\end{align*}
	    over $f^{[m]\ast}(V)$,
		where $\omega_{X^{[m]}/Y}$ is the relative dualizing sheaf. Additionally, there is an injective morphism given by the trace map:
		\begin{align} \label{align_trace_map}
			{\rm tr}:\rho_\ast(\omega_{X^{(m)}/Y})\to \omega_{X^{[m]}/Y}.
		\end{align}
		By applying the flat base change theorem, we obtain an isomorphism in the derived category $D(V)$:
		\begin{align*}
			Rf^{[m]}_\ast(\omega_{X^{[m]}/Y}\otimes \sE^{[m]})|_V\simeq \bigotimes^{m}Rf_\ast(\omega_X\otimes \sE)|_V.
		\end{align*}
		Since every $R^qf_\ast(\omega_X\otimes \sE)|_V$ is locally free, it follows that $\bigotimes^{m}R^qf_\ast(\omega_X\otimes \sE)|_V$ is isomorphic to a direct summand of $R^{mq}f^{[m]}_\ast(\omega_{X^{[m]}/Y}\otimes \sE^{[m]})|_V$.
		Hence, we have a surjective morphism
		\begin{align*}
			R^{mq}f^{[m]}_\ast(\omega_{X^{[m]}/Y}\otimes \sE^{[m]})|_V\to \bigotimes^{m}R^qf_\ast(\omega_X\otimes \sE)|_V.
		\end{align*}
		By taking reflexive hulls, we obtain a morphism:
		\begin{align}\label{align_map_alpha}
			\alpha: R^{mq}f^{[m]}_\ast(\omega_{X^{[m]}/Y}\otimes \sE^{[m]})\to \bigotimes^{[m]}R^qf_\ast(\omega_X\otimes \sE)\to S^{[m]}R^qf_\ast(\omega_X\otimes \sE),
		\end{align}
		which is surjective over $V$. Here, $\otimes^{[m]}R^qf_\ast(\omega_X\otimes \sE)$ denotes the reflexive hull of $\otimes^{m}R^qf_\ast(\omega_X\otimes \sE)$.
		
	Let $h_{\sE^{[m]}}=\otimes_{i=1}^{m}p_i^\ast (hh_Fh_L)$, $h^{[m]}_0=\otimes_{i=1}^{m}p_i^\ast (h_0)$, $h_{\sE^{(m)}}=\rho^\ast(h_{\sE^{[m]}})$, and $h^{(m)}_0=\rho^\ast(h^{[m]}_0)$. Then $h^{(m)}_0$ is a smooth Hermitian metric on $\sE^{(m)}$, and $h_{\sE^{(m)}}$ is a singular Hermitian metric on $\sE^{(m)}$ satisfying $h^{(m)}_0\lesssim h_{\sE^{(m)}}$. Therefore, we have the following inclusion:
		\begin{align*}
			S_{X^{(m)}}(\sE^{(m)},h_{\sE^{(m)}})\subset S_{X^{(m)}}(\sE^{(m)},h^{(m)}_0)\simeq\omega_{X^{(m)}}\otimes\sE^{(m)}.
		\end{align*}
		By combining this with the trace map (\ref{align_trace_map}), we obtain an injective morphism:
		\begin{align}\label{align_inj}
			\beta:\rho_\ast(S_{X^{(m)}}(\sE^{(m)},h_{\sE^{(m)}}))\to \omega_{X^{[m]}/Y}\otimes\sE^{[m]}\otimes f^{[m]\ast}(\omega_Y).
		\end{align}
		Since $hh_Fh_L$ is Nakano semi-positive over some Zariski open subset of $X$ and tame on $X$ (Lemma \ref{lem_Nsp_tame}), $h_{\sE^{(m)}}$ is Nakano semi-positive over some Zariski open subset of $X^{(m)}$ and tame on $X^{(m)}$. Applying Theorem \ref{thm_abstract_Kollar_package}, we have:
		$$\rho_\ast(S_{X^{(m)}}(\sE^{(m)},h_{\sE^{(m)}}))\simeq R\rho_\ast(S_{X^{(m)}}(\sE^{(m)},h_{\sE^{(m)}})).$$
		Therefore, (\ref{align_inj}) induces a morphism:
		\begin{align}\label{align_beta'}
			\beta':R^{mq}f^{(m)}_\ast(S_{X^{(m)}}(\sE^{(m)},h_{\sE^{(m)}}))\to R^{mq}f^{[m]}_\ast(\omega_{X^{[m]}/Y}\otimes \sE^{[m]})\otimes\omega_Y,
		\end{align}
		where $f^{(m)}$ denotes $f^{[m]}\circ\rho$.
	By applying Corollary \ref{cor_L2_interpretation} on $X^{[m]o}$, it can be seen that $\beta$ is also an isomorphism over $X^{[m]o}=(f^{[m]})^{-1}(U)$. Therefore, $\beta'$ is an isomorphism over $U$. By combining (\ref{align_beta'}) with (\ref{align_map_alpha}), we obtain a morphism
	\begin{align}\label{align_Viehweg_map}
		R^{mq}f^{(m)}_\ast(S_{X^{(m)}}(\sE^{(m)},h_{\sE^{(m)}}))\to S^{[m]}R^qf_\ast(\omega_X\otimes \sE)\otimes\omega_Y
	\end{align}
	which is surjective over $U\cap V$.
	
		\emph{Step 3: Castelnuovo-Mumford's criterion.} 
		Applying Theorem \ref{thm_abstract_Kollar_package} to $S_{X^{(m)}}(\sE^{(m)},h_{\sE^{(m)}})$ yields the following result:
		$$H^i(Y,R^{mq}f^{(m)}_\ast(S_{X^{(m)}}(\sE^{(m)},h_{\sE^{(m)}}))\otimes A_Y^{\dim Y+1}\otimes A_Y^{-i})=0,\quad\text{for all } i>0.$$
		This means that $R^{mq}f^{(m)}_\ast(S_{X^{(m)}}(\sE^{(m)},h_{\sE^{(m)}}))\otimes A_Y^{\dim Y+1}$ is $0$-regular and therefore generated by global sections. Combining this with the generic surjective map (\ref{align_Viehweg_map}), we can conclude that 
		$S^{[m]}R^qf_\ast(\omega_{X/Y}\otimes \sE)\otimes\omega_Y\otimes A_Y^{\dim Y+1}$
		is generated by global sections at every point of $V\cap U$. This demonstrates that $R^qf_\ast(\omega_{X/Y}\otimes \sE)$ is weakly positive.
	\end{proof}
	\subsection{Generic vanishing theorem}
	\begin{defn}
		Let $A$ be an abelian variety. A coherent sheaf $F$ on $A$ is called a \emph{GV-sheaf} if 
		$${\rm codim}_{{\rm Pic}^0(A)}\left\{M\in{\rm Pic}^0(A)\mid H^i(A,F\otimes M)\neq0\right\}\geq i$$
		for every $i$.
	\end{defn}
Let us recall Hacon's criterion of a GV-sheaf \cite{Hacon2004} (see also \cite[Theorem 25.5]{Schnell_GV} and \cite{Popa2011}).
\begin{lem}\label{lem_GV_lemma}
	Suppose that for every finite \'etale morphism $\varphi:B\to A$ of abelian varieties, and every ample line bundle $L$ on $B$, one has $H^i(B,\varphi^*(F)\otimes L)=0$ for $i>0$. Then $F$ is a GV-sheaf.
\end{lem}
	\begin{thm}
		Notations as in Theorem \ref{thm_main}. Let $f:X\to A$ be a morphism to an abelian variety. Then $R^qf_\ast(\omega_X\otimes(P{_{D-N,(2)}}(H)\cap j_\ast K)\otimes F\otimes L)$ is a GV-sheaf for every $q\geq0$. As a consequence, 
		$${\rm codim}_{{\rm Pic}^0(A)}\left\{M\in{\rm Pic}^0(A) \mid H^i(X,\omega_X\otimes(P{_{D-N,(2)}}(H)\cap j_\ast K)\otimes F\otimes L\otimes f^\ast M)\neq0\right\}$$
		$$\geq i-(\dim X-\dim f(X))$$
		for all $i$.
	\end{thm}
	\begin{proof}
		Let $\varphi:B\to A$ be a finite \'etale morphism between abelian varieties. We define $Z:=X\times_AB$. Then, we have the following commutative diagram:
		\begin{align*}
			\xymatrix{
				Z\ar[r]^\psi\ar[d]^g & X \ar[d]^f \\
				B\ar[r]^\varphi & A
			}.
		\end{align*}
		Since $\psi$ is \'etale, $D':=\psi^\ast(D)$ is a reduced simple normal crossing divisor on $Z$. Moreover, we have the following isomorphism:
		\begin{align*}
			\psi^\ast(P_{D-N,(2)}(H)\cap j_\ast K) \simeq P_{D'-\psi^\ast N,(2)}(\psi^\ast H)\cap j'_\ast \psi^\ast K,
		\end{align*}
		where $j':Z\backslash \psi^\ast(D)\to Z$ is the open immersion. This implies that there is an isomorphism:
		\begin{align*}
			&\varphi^\ast R^qf_\ast(\omega_X\otimes(P_{D-N,(2)}(H)\cap j_\ast K)\otimes F\otimes L)\\\nonumber
			\simeq &R^qg_\ast(\omega_Z\otimes(P_{D'-\psi^\ast N,(2)}(\psi^\ast H)\cap j'_\ast \psi^\ast K)\otimes \psi^\ast F\otimes \psi^\ast L).
		\end{align*}
		Consequently, using Theorem \ref{thm_main}, the conditions in Lemma \ref{lem_GV_lemma} are satisfied.
		
		This establishes the first claim of the theorem. The second claim can be derived by applying the first claim to the spectral sequence:
		$$E_2^{p,q}:=H^q(A,R^pf_\ast(\omega_X\otimes(P{_{D-N,(2)}}(H)\cap j_\ast K)\otimes F\otimes L)\otimes M)$$
	$$\Rightarrow H^{p+q}(X,\omega_X\otimes(P{_{D-N,(2)}}(H)\cap j_\ast K)\otimes F\otimes L\otimes f^\ast M)$$
		where $M\in \operatorname{Pic}^0(A)$.
		It should be noted that $R^pf_\ast(\omega_X\otimes(P_{D-N,(2)}(H)\cap j_\ast K)\otimes F\otimes L)=0$ for every $p>\dim X-\dim f(X)$, as mentioned in Theorem \ref{thm_main}.
	\end{proof}
	\bibliographystyle{plain}
	\bibliography{CGM_Kollar}
	
\end{document}